\documentclass[11pt]{article}

\usepackage[margin=1in]{geometry}
\usepackage{amsmath,amssymb,bm}
\usepackage{microtype,booktabs,array}
\usepackage{mathrsfs,mathabx}
\usepackage{amsthm}
\usepackage{hyperref}
\usepackage{graphicx}
\usepackage{enumerate} 
\usepackage{enumitem}
\numberwithin{equation}{section} 
\usepackage{csquotes}
\usepackage{todonotes}
\usepackage{mathtools}

\usepackage[comma,sort, numbers]{natbib}

\DeclareMathOperator{\interior}{int}

\setlist[itemize]{topsep=4pt,itemsep=2pt,leftmargin=1.7em}
\setlist[enumerate]{topsep=4pt,itemsep=3pt,leftmargin=2em}
\definecolor{propRed}{HTML}{A84343}
\definecolor{propBlue}{HTML}{2968A3}
\definecolor{propTeal}{HTML}{267C79}
\definecolor{propAmber}{HTML}{AD7A24}
\definecolor{propPurple}{HTML}{75649A}
\tikzset{propvertex/.style={circle,fill=propRed,inner sep=1.8pt},
  propdual/.style={circle,fill=propTeal,inner sep=1.7pt},
  propzero/.style={circle,draw=propTeal,fill=propTeal!9,
    minimum size=6.5mm,inner sep=0pt,font=\small},
  propone/.style={circle,draw=propAmber,fill=propAmber!13,
    minimum size=6.5mm,inner sep=0pt,font=\small},
  propedge/.style={line width=.65pt},
  propdualedge/.style={draw=propTeal,dashed,line width=.85pt},
  propport/.style={draw=black!55,densely dashed,line width=.7pt}}

\definecolor{hred}{RGB}{174,40,54}
\definecolor{hblue}{RGB}{32,91,157}
\definecolor{hteal}{RGB}{0,118,116}

\usepackage[utf8]{inputenc}
\usepackage{graphicx,tikz,tcolorbox} 
\usetikzlibrary{arrows.meta}
\usetikzlibrary{calc,decorations.pathreplacing}
\usepackage{placeins}

\newtheorem{theorem}{Theorem}[section]
\newtheorem{observation}[theorem]{Observation}

\newtheorem{lemma}[theorem]{Lemma}
\newtheorem{proposition}[theorem]{Proposition} 
\newtheorem{question}[theorem]{Question}

\theoremstyle{definition}
\newtheorem{definition}[theorem]{Definition}

\newtheorem{remark}[theorem]{Remark}

\allowdisplaybreaks

\title{ Almost Empty Monochromatic Polygons in Planar Point Sets }

\author{Bhaswar B. Bhattacharya\thanks{Department of Statistics and Data Science, University of Pennsylvania and National University of Singapore,  \texttt{bhaswar@wharton.upenn.edu}} \and 
Sandip Das\thanks{Advanced Computing and Microelectronics Unit, Indian Statistical Institute, Kolkata, \texttt{sandip.das.69@gmail.com}}
\and 
Sk Samim Islam\thanks{School of Computer Science and Engineering, University of Electronic Science and Technology of China, Chengdu, \texttt{samimislam08@gmail.com}}
\and 
Aashirwad Mohapatra\thanks{Advanced Computing and Microelectronics Unit, Indian Statistical Institute, Kolkata, \texttt{aashirwad\_r@isical.ac.in}} 
\and 
Saumya Sen\thanks{Advanced Computing and Microelectronics Unit, Indian Statistical Institute, Kolkata, \texttt{saumyasen72@gmail.com}}
} 
\date{}

\begin{document}

\maketitle

\begin{abstract}  
For integers $r \ge 3$, $c \ge 2$, and $s \ge 0$, let $\mathsf{M}_r^{\star}(c,s)$ denote the smallest integer such that any set of at least $\mathsf{M}_r^{\star}(c,s)$ points in the plane in general position (no three on a line), colored with $c$ colors, contains a monochromatic $r$-gon (not necessarily convex) with at most $s$ interior points. Define $\lambda_r^{\star}(c)$ to be the smallest integer such that $\mathsf{M}_r^{\star}(c,\lambda_r^{\star}(c)) < \infty$. The case of triangles, that is, $r=3$, has been studied in recent years. In this paper, we establish bounds on $\lambda_r^{\star}(c)$, for $r \ge 4$. Specifically,  in the bichromatic case, we prove that $\lambda_r^{\star}(2)\leq \lfloor\frac{r-2}{2}\rfloor$, for every $r\geq3$. For hexagons, we sharpen this bound to $\lambda_6^{\star}(2)\leq1$, that is, every sufficiently large bichromatic planar point set in general position contains a monochromatic hexagon with at most one interior point.
Next, for any number of colors $c\geq3$ and $r\geq3$, we show that 
$$(r - 2) \left \lfloor \frac{c-1}{2} \right \rfloor \leq \lambda_r^\star(c)
\leq (r-2)(c-2)+\left\lfloor(r-2)\delta_c\right\rfloor,  $$
for some constant $\delta_c \in (0, \frac{1}{2}]$, depending on $c$. The upper bounds combine spiral polygonizations with nested convex-hull arguments and, in the hexagon case, incidence counting in $4$-connected triangulations. The lower bounds follow from colored Horton-set constructions.  
\end{abstract}


\section{ Introduction }
\label{sec:introduction}

The celebrated Erd\H{o}s-Szekeres theorem asserts that for every positive integer $r \geq 3$, there exists a smallest integer $\mathsf{ES}(r)$ such that any set of at least $\mathsf{ES}(r)$ points in the plane in general position (that is, no three on a line), contains a convex $r$-gon, that is, $r$ points which lie on the vertices of a convex polygon. In particular, \citet{erdos1935combinatorial,erdHos1960some} established the following bounds: 
$$2^{r - 2}+1 \leq \mathsf{ES}(r) \leq {{2r - 4} \choose {r - 2}} + 1.$$ 
They also famously conjectured that the lower bound is sharp. 
After a series of improvements, in a major breakthrough \citet{suk2017polygon} 
almost resolved this conjecture by showing that
$$\mathsf{ES}(r) = 2^{r+O(r^{\frac{2}{3}} \log r )}.$$ 
Later, \citet{holmsen2020two} improved the exponent to $r+O(\sqrt{r \log r} )$. The conjecture of Erd\H{o}s and Szekeres has also been verified for a few small values of $r$. In particular, it is known that $\mathsf{ES}(3) = 3$, $\mathsf{ES}(4) = 5$ \cite{kalbfleisch1970combinatorial}, $\mathsf{ES}(5) = 9$ \cite{kalbfleisch1970combinatorial}, and $\mathsf{ES}(6) = 17$ \cite{szekeres2006computer}.

A strengthening of the above question, formulated by Erd\H{o}s in \cite{erdos1978some}, asks whether every sufficiently large planar point set in general position contains an $r$-{\it hole}, that is, a convex $r$-gon whose interior contains no other points of the set. Denote by $\mathsf{H}(r)$, if it exists, the smallest integer such that every set of at least $\mathsf{H}(r)$ points in the plane in general position contains
an $r$-hole. Esther Klein showed $\mathsf{H}(4) = 5$ and Harborth \cite{harborth1978konvexe} proved $\mathsf{H}(5) = 10$. A few years later, Horton \cite{horton1983sets}  showed that it is possible to construct an arbitrarily large set of points without a 7-hole, which shows $\mathsf{H}(r)$ does not exist for $r \geq 7$. The existence of $\mathsf{H}(6)$ remained open for several years, until Gerken \cite{gerken2008empty} and Nicol{\'a}s \cite{nicolas2007empty} independently established the finiteness of $\mathsf{H}(6)$ (see also \cite{valtr2008empty}). In particular, Gerken \cite{gerken2008empty} showed that $\mathsf{H}(6)  \leq \mathsf{ES}(9)$, which  was later improved to  $\mathsf{H}(6) \leq \max\{\mathsf{ES}(8), 400\}$ by 
Koshelev \cite{koshelev2007erdos}. For the lower bound, Overmars \cite{overmars2002polygon} constructed a set of 29 points without any 6-hole, which implies $\mathsf{H}(6) \geq 30$. In a recent breakthrough, \citet{heule2024happy} showed that $\mathsf{H}(6) = 30$, finally resolving the 6-hole problem.

The non-existence of \(\mathsf{H}(r)\) for \(r \ge 7\) has motivated the study of various relaxed variants. One such relaxation, initiated by \citet{nyklova2003almost}, is the notion of {\it almost empty polygons}, that is, convex polygons containing few interior points (which is also closely related to the notion of islands \cite{bautista2011computing,balko2022holes}). Specifically, given integers $r \geq 3$ and $s \geq 0$, let $\mathsf{H}(r, s)$ denote the minimum integer such that any set of at least $\mathsf{H}(r,s)$ points in the plane in general position contains a convex $r$-gon with at most $s$ interior points. Obviously, $\mathsf{ES}(r) \leq \mathsf{H}(r,s) \leq \mathsf{H}(r ,0) = \mathsf{H}(r)$. Using a Horton set construction, Nyklov{\'a} \cite{nyklova2003almost} showed that $\mathsf{H}(r, s)$ does not exist when $s \leq 2^{\frac{r+6}{4}} - r - 3$. Later, \citet{koshelev2011computer} studied the almost empty hexagon problem and, more recently, 
\citet{huemer2022weighted} proved several equalities involving weighted sums of the numbers of convex $r$-gons with exactly $s$ interior points (generalizing results of \cite{pinchasi2006empty} on empty convex $r$-gons). A weaker restriction where the number of interior points is divisible by a certain number has been studied in \cite{bialostocki1991some}.

In another direction, Devillers et al. \cite{devillers2003chromatic} considered colored variants of the Erd\H{o}s-Szekeres problem. In such problems, a set of points is partitioned into $c \geq 2$ color classes, and a polygon spanned by points in the set is called {\it monochromatic} if all of its vertices are of the same color. In this setting, Grima et al. \cite{grima2009some} showed that any set of 9 points colored arbitrarily with 2 colors contains a monochromatic 3-hole. On the other hand, Devillers et al. \cite{devillers2003chromatic} showed that there exist arbitrarily large 3-colored point sets with no monochromatic 3-hole and there exist arbitrarily large 2-colored point sets with no monochromatic 5-hole. Devillers et al. \cite{devillers2003chromatic} also conjectured that any sufficiently large 2-colored point set contains a monochromatic 4-hole. While this conjecture remains open, a relaxation which does not require the 4-hole to be convex has been established by Aichholzer et al. \cite{aichholzer2010large}.  Other chromatic variations, such as balanced and rainbow holes, have also been studied (see, for example, \cite{arevalo2022rainbow,diaz2021note,fabila2021empty,kano2021discrete} and the references therein).

Combining the notions of almost empty polygons and monochromatic holes, one can consider the problem of existence of almost empty monochromatic polygons. Formally, given integers $r \geq 3$, $c \geq 2$, and $s \geq 0$, define $\mathsf{M}_r(c,s)$ to be the least integer such that any set of at least $\mathsf{M}_r(c,s)$  points in the plane in general position colored with $c$ colors contains a monochromatic convex $r$-gon with at most $s$ interior points. Moreover, denote by $\lambda_r(c)$ the least integer such that $\mathsf{M}_r(c, \lambda_r(c)) < \infty$. \citet{colorempty} considered the case of triangles (that is, $r=3$) and proved that, for any $c \geq 2$, 
\begin{align}\label{eq:almostempty3}
\left\lfloor \frac{c-1}{2} \right\rfloor \leq \lambda_3(c) \leq c-2. 
\end{align}
For $c \geq 4$, the upper bound was improved to $c-3$ by \citet{cravioto2019almost}. Recently, the authors further improved the upper bound to 
$\lambda_3(c) \leq c-\sqrt{c\log c}+o (\sqrt{c\log c} )$, for all sufficiently large $c$ \cite{interiorpoints}. In their paper, \citet{cravioto2019almost} also considered the case of quadrilaterals (that is, $r=4$) and showed that $\lambda_4(c) \leq 2c-3$, for $c \geq 2$ (see \cite[Theorem 3]{cravioto2019almost}). Relaxing the convexity requirement, one can define $\mathsf{M}_r^{\star}(c,s)$ to be the least integer such that any set of at least $\mathsf{M}_r^{\star}(c,s)$  points in the plane in general position colored with $c$ colors contains a monochromatic $r$-gon (not necessarily convex) with at most $s$ interior points. Further, denote by $\lambda_r^{\star}(c)$ the least integer such that $\mathsf{M}_r^{\star}(c, \lambda_r^{\star}(c)) < \infty$. Clearly, $\lambda_3^{\star}(c) = \lambda_3(c)$ and $\lambda_r^{\star}(c) \leq \lambda_r(c)$, for any $r \geq 4$. Moreover, in this notation the result of Aichholzer et al. \cite{aichholzer2010large} can be stated as $\lambda_4^{\star}(2)=0$. Liu and Zhang \cite{liu2018almost} proved that  for $c \geq 3$, 
\begin{align}\label{eq:almostempty4}
2 \Big \lfloor  \frac{c - 1}{2} \Big \rfloor  \leq  \lambda _4^{\star}(c) \leq \lambda_4(c) \leq 2c - 3 . 
\end{align}
They also showed that $\mathsf{M}_4^{\star}(2, 1) = 11$ and $\mathsf{M}_4^{\star}(3, 2) \leq 120$. 

In this paper, we initiate a systematic study of the existence of monochromatic $r$-gons with few interior points in colored point sets. Specifically, we establish bounds on $\lambda_r^{\star}(c)$ for all $c \geq 2$ and $r \geq 3$. The proofs combine several geometric and combinatorial techniques, including spiral polygonizations and nested convex-hull arguments. A new technical ingredient is an incidence-counting argument for $4$-connected triangulations, which may be of independent
interest.

\subsection{Two Colors}

Our first theorem gives a general upper bound in the bichromatic case.  

\begin{theorem}\label{thm:c2} 
For any integer $r \geq 3$, $\lambda_r^{\star}(2) \leq \left\lfloor\frac{r-2}{2} \right\rfloor$.
\end{theorem}

The proof of Theorem~\ref{thm:c2} is given in Section~\ref{sec:thm1pf}. The proof proceeds by constructing nested convex hulls of alternating colors. A key ingredient is the spiral
polygonization result of \citet{iwerks2011spiral}. The spiral structure allows triangular regions to be grouped consecutively into simple $r$-gons, giving many pairwise interior-disjoint $r$-gons containing no other points of their own color. We apply this construction first to the points of color 1 and then to the color 2 points lying inside their convex hull. If the desired monochromatic $r$-gon does not exist, each polygon must contain sufficiently many points of the opposite color. Interior-disjointness forces lower bounds on the number of points in successive layers. Starting inside a sufficiently large monochromatic convex polygon, whose existence follows from the Erd\H{o}s--Szekeres theorem, makes these lower bounds incompatible with the loss of boundary points when we return to the original color.

\begin{remark}  
For the first few values of $r$, Theorem~\ref{thm:c2} has the following consequences:

\begin{itemize}

\item For $r=3$, Theorem~\ref{thm:c2} gives $\lambda_3^{\star}(2)=0$, recovering the result of
Grima et al.~\cite{grima2009some} that every sufficiently large $2$-colored point set contains a monochromatic empty triangle.  

\item For $r=4$, Theorem~\ref{thm:c2} gives $\lambda_4^{\star}(2)\leq1$. In this case, however, the stronger result $\lambda_4^{\star}(2)=0$ is already known, due to Aichholzer et al.~\cite{aichholzer2010large}.

\item For $r=5$, the first case yielding a new bound, Theorem~\ref{thm:c2} gives $\lambda_5^{\star}(2)\leq1$. Thus, every sufficiently large $2$-colored point set contains a monochromatic pentagon with at most one interior point. It remains open whether $\lambda_5^{\star}(2)=0$, or equivalently, if 
every sufficiently large $2$-colored point set contains a monochromatic empty pentagon (see Section~\ref{sec:conclusions} for further discussion). Recall that \cite[Theorem~3.4]{devillers2003chromatic} shows that there exist arbitrarily large two-colored point sets with no monochromatic empty \emph{convex} 5-gon. Nevertheless, monochromatic empty 5-gons without the convexity requirement may still exist.

\item  For $r=6$, Theorem~\ref{thm:c2} gives $\lambda_6^{\star}(2)\leq2$. In the next theorem, we strengthen this to $\lambda_6^{\star}(2)\leq1$, showing that every sufficiently large $2$-colored planar point set in general position contains a monochromatic hexagon with at most one interior point. Obtaining this sharper bound requires a substantially more delicate argument and several new techniques, as explained below.

\end{itemize}
\end{remark}

\begin{theorem}\label{thm:hexagon} 
 $\lambda_6^{\star}(2) \leq 1$. In particular, $\mathsf{M}_6^{\star}(2, 1) \leq 306447$.  
\end{theorem}

Theorem \ref{thm:hexagon} shows that every sufficiently large 2-colored planar point set in general position contains a monochromatic hexagon with at most one interior point. The proof is given in Section \ref{sec:hexagonpf}. We provide a high-level outline of the proof below. Throughout, $S = S_1 \cup S_2$ will be a 2-colored point set, where $S_1$ and $S_2$ denote the sets of points of color $1$ and $2$, respectively. Assume, without loss of generality, $|S_1| \geq |S_2|$, that is, 1 is the more frequent color.

\begin{itemize}  

\item The first step establishes that if the convex hull of $S_1$ has sufficiently many vertices, then $S$ contains a monochromatic hexagon with at most one interior point (Proposition~\ref{prop:lhull}). To prove this, we adapt an argument of \citet{aichholzer2010large}, used to establish the existence of empty monochromatic quadrilaterals in bichromatic point sets, in our context. The argument proceeds by iteratively constructing nested convex hulls, alternating between the two colors. At each stage, we construct many pairwise interior-disjoint monochromatic hexagons containing no points of their own color. If the desired hexagon does not exist, each of these hexagons must contain at least two points of the opposite color. Iterating this requirement forces the nested point sets to remain large. These lower bounds become eventually incompatible with the deficit in the hull vertices, giving the required contradiction.

\item The small-hull case is more delicate. Here, the starting point is a $4$-connected triangulation of $S_1$. First, invoking a result of \citet{DGR}, we show that such a triangulation always exists, whenever $S_1$ is large enough and $S$ does not contain a monochromatic hexagon with at most one interior point. The existence of such a triangulation ensures that four triangular faces joined in a chain or around a central face form a simple hexagon of color 1 with no point of color 1 inside. Then we weight each face of the triangulation by the number of points of color 2 in its interior. Every such group of four faces which induce a tree in the dual graph must have total weight at least two, if the desired hexagon does not exist.  This restricts how empty faces can cluster together and meet faces containing only one point of color 2. Counting these adjacencies relates the number of empty faces to the available color 2 points and the hull boundary. The remaining exceptional clusters, consisting of four empty faces surrounding a vertex, are controlled by a local argument using edge flips and angle constraints. Together, these bounds contradict the small-hull assumption (see Proposition~\ref{prop:shull}).
\end{itemize}

Determining the exact value of $\lambda_6^{\star}(2)$, and more generally of $\lambda_r^{\star}(2)$ for $r\geq5$, remains open. No positive lower bounds are currently known for these quantities. Indeed, it is possible that $\lambda_r^{\star}(2)=0$ for every fixed $r\geq3$, that is, every sufficiently large $2$-colored planar point set in general position contains a monochromatic empty $r$-gon, not necessarily convex. The following result establishes this assertion for Horton sets. Since Horton sets are a principal source of counterexamples in such problems, this provides evidence in favor of the possibility that $\lambda_r^{\star}(2)=0$. The proof is given in Appendix~\ref{sec:Hpf} (see also Section~\ref{sec:conclusions}
for related discussion).

\begin{proposition}\label{prop:H}
For any fixed integer $r\geq3$, every $2$-colored Horton set with at least $4^{r-1}$ points contains a monochromatic empty $r$-gon.
\end{proposition}

\subsection{Three or More Colors}

For $c \geq 3$, we establish the following bounds on $\lambda_r^{\star}(c)$.\footnote{A preliminary version of some of the results in this section appeared
in the proceedings of the {\it Conference on Algorithms and Discrete Applied Mathematics} (CALDAM), 77--87, 2025.  }

\begin{theorem}\label{thm:r}  For all integers $r \geq 3$ and $c \geq 3$, 
\begin{align*}
(r - 2) \left \lfloor \frac{c-1}{2} \right \rfloor \leq  \lambda_r^{\star}(c) \leq  (r-2)(c-2)+\left\lfloor(r-2)\delta_c\right\rfloor  , 
\end{align*}
where $\delta_c:=\frac{2}{2c-3+\sqrt{(2c-3)^2+8}} \in (0, \frac{1}{2}]$.  
\end{theorem}

The proof of Theorem~\ref{thm:r} is given in Section \ref{sec:rpf}. The upper bound builds on the spiral polygonization and nested convex-hull approach used in Theorem~\ref{thm:c2}. In the bichromatic case, constructing interior-disjoint $r$-gons of one color forces points of the opposite color inside them, allowing a direct comparison between the two color classes. With more colors, these interior points may be distributed among several classes. The additional ingredient is to apply the bounds on the number of interior-disjoint $r$-gons to every remaining color class and combine the resulting inequalities to obtain a contradiction. The lower bound follows from a colored Horton-set construction. The proof of the upper bound also applies when $c=2$. Note that substituting $c=2$ recovers the bound $\lambda_r^\star(2)\leq\lfloor \frac{r-2}{2}\rfloor$ in Theorem~\ref{thm:c2}, since $\delta_2=\frac{1}{2}$. We presented the bichromatic case first, with a direct proof to emphasize the result and illustrate the main ideas in a simpler setting.

\begin{remark}
For the first few values of $r$, Theorem~\ref{thm:r} has the following consequences:
\begin{itemize}

\item For $r=3$, Theorem~\ref{thm:r} recovers the bounds in \eqref{eq:almostempty3}.

\item For $r=4$ and $c\geq3$, Theorem~\ref{thm:r} gives
\begin{equation}\label{eq:ae4}
2\left\lfloor\frac{c-1}{2}\right\rfloor \leq \lambda_4^\star(c)\leq 2c-4,
\end{equation}
improving the upper bound in \eqref{eq:almostempty4} by one. Interestingly, although not covered by our results, the upper bound in \eqref{eq:ae4} also holds for $c=2$, since Aichholzer et al.~\cite{aichholzer2010large} proved that every sufficiently large bichromatic planar point set in general position contains an empty monochromatic quadrilateral, not necessarily convex. Hence, $\lambda_4^\star(2)=0$. Also, for $c=3$, the bounds in \eqref{eq:ae4} coincide to give $\lambda_4^\star(3)=2$. 

\item For $r=5$ and $c\geq3$, Theorem~\ref{thm:r} gives, 
\begin{equation*} 
3\left\lfloor\frac{c-1}{2}\right\rfloor \leq \lambda_5^\star(c)\leq 3c-6.
\end{equation*}
In particular, the bounds coincide for $c=3$, giving $\lambda_5^\star(3)=3$.
Thus, every sufficiently large $3$-colored planar point set in general position contains a monochromatic pentagon, not necessarily convex, with at most three interior points.
Further, the bound of three interior points cannot be reduced to two.  

\end{itemize}
\end{remark}

\section{Proof of Theorem \ref{thm:c2} 
\label{sec:thm1pf}}

Let $S$ be a set of points in the plane in general position. We denote the convex hull of $S$ by $CH(S)$ and the vertex set of the convex hull by $V(CH(S))$. A \emph{polygonization} of $S$ is a simple polygon whose vertex set is $S$. A triangulation of a polygon is said to be {\it serpentine} if its dual graph is a path.\footnote{In a dual graph of a triangulation $T$, the vertices are the bounded triangular faces of $T$ and two triangles are adjacent in the dual if they share an interior edge.}
A polygonization is called a \emph{spiral} if its boundary consists of exactly one chain of convex vertices and at most one chain of reflex
(nonconvex) vertices (see Figure~\ref{fig:s}). \citet{iwerks2011spiral} showed that every planar set of $n$ points in general position has a spiral polygonization that admits a serpentine triangulation, which can be computed in $O(n\log n)$ time. The following lemma is a consequence of this result:

\begin{lemma}[\cite{iwerks2011spiral}]\label{lem:r}  
Fix $r \geq 3$ and suppose $S$ is a set of points in the plane in general position, with $|S| \geq 3$. Then $S$ contains at least 
\begin{align}\label{eq:rpoints}
\left\lfloor \frac{|S| - 2}{r - 2}\right\rfloor + \left\lfloor \frac{|S| - |V(CH(S))|}{r-2}  \right\rfloor 
\end{align} 
pairwise interior-disjoint empty $r$-gons.
\end{lemma}

\begin{proof}  
If $|S|=|V(CH(S))|$, triangulate the convex polygon $CH(S)$ by a fan and group consecutive
triangles into blocks of size $r-2$. Each block forms an empty $r$-gon and the blocks are pairwise interior-disjoint. This gives the required bound in \eqref{eq:rpoints}. 

Now, assume $|S| >  |V(CH(S))|$. Apply the construction in \cite{iwerks2011spiral} to obtain
a spiral polygonization $\mathcal P^+$ of $S$. Note that $\mathcal P^+$ contains all but one hull edge, and the closure of $CH(S)\setminus\mathcal P^+$ is a simple polygon $\mathcal P^-$ (as shown in Figure~\ref{fig:s}). The vertices of $\mathcal P^-$ are the $|S|-|V(CH(S))|$ interior points of $S$ together with the endpoints of the missing hull edge, that is, $\mathcal P^-$ has a total of $|S|-|V(CH(S))|+2$ vertices. Further, $\mathcal P^-$ is also a spiral: at each non-hull vertex, convexity and reflexivity are interchanged between the two polygons. It is known that every spiral polygon admits a serpentine triangulation \cite[Lemma~4.1]{orourke1987artgallery}. Choose such
triangulations of $\mathcal P^+$ and $\mathcal P^-$, and order their triangles along the corresponding dual paths. The two triangulations contain $|S|-2$ and $|S|-|V(CH(S))|$ triangles, respectively. A block of $r-2$ consecutive triangles along each of these triangulations forms a simple $r$-gon. Hence, discarding any leftover final blocks, the result in \eqref{eq:rpoints} follows.  
\end{proof}

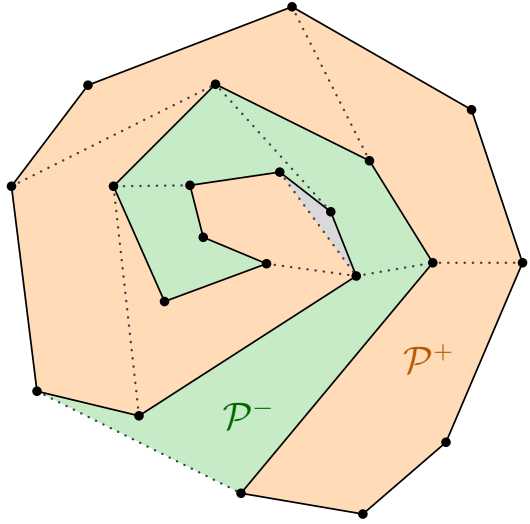
\begin{figure}[htbp]
\centering
\begin{tikzpicture}[
  x=.0125cm,y=-.0125cm,
  boundary/.style={draw=black,line width=.65pt,line join=round},
  separator/.style={draw=black!75,line width=.9pt,
                    dash pattern=on 0pt off 3pt,line cap=round},
  spirallabel/.style={font=\Large,inner sep=1pt}
]
\foreach \i/\x/\y in {
   0/396/330, 1/329/302, 2/315/247, 3/410/233,
   4/464/275, 5/491/343, 6/261/491, 7/153/465,
   8/126/248, 9/207/141,10/423/58,11/613/167,
  12/667/329,13/586/519,14/369/573,15/572/329,
  16/505/221,17/342/140,18/234/248,19/288/370,
  20/498/595}
  \coordinate (v\i) at (\x,\y);

\fill[orange,fill opacity=.28]
  (v0)--(v1)--(v2)--(v3)--(v4)--(v5)--(v6)
  --(v7)--(v8)--(v9)--(v10)--(v11)--(v12)
  --(v13)--(v20)--(v14)--(v15)--(v16)--(v17)--(v18)
  --(v19)--cycle;

\fill[green!65!black,fill opacity=.22]
  (v7)--(v6)--(v5)--(v4)--(v3)--(v2)--(v1)
  --(v0)--(v19)--(v18)--(v17)--(v16)--(v15)
  --(v14)--cycle;

\fill[gray!30] (v3)--(v4)--(v5)--cycle;

\draw[boundary]
  (v0)--(v1)--(v2)--(v3)--(v4)--(v5)--(v6)
  --(v7)--(v8)--(v9)--(v10)--(v11)--(v12)
  --(v13)--(v20)--(v14)--(v15)--(v16)--(v17)--(v18)
  --(v19)--cycle;

\foreach \i/\j in {12/15,10/16,8/17,6/18,0/5}
  \draw[separator] (v\i)--(v\j);
\draw[separator] (v3)--(v5);

\foreach \i/\j in {2/18,4/17,5/15}
  \draw[separator] (v\i)--(v\j);
   \draw[separator] (v7)--(v14);

\foreach \i in {0,...,20}
  \fill[black] (v\i) circle[radius=1.8pt];
\node[spirallabel,text=orange!70!black] at (568,430)
  {$\mathcal{P}^{+}$};
\node[spirallabel,text=green!40!black] at (377,489)
  {$\mathcal{P}^{-}$};
\end{tikzpicture}
\caption{\small{The orange polygon $\mathcal{P}^{+}$ is a spiral polygonization. The complementary spiral polygon $\mathcal{P}^{-}$ is shown in green. The dotted diagonals mark pentagons, each formed by three consecutive triangles. The leftover triangle at the centre of the spiral is shown in grey.  }}
\label{fig:s}
\end{figure}

\begin{remark} 
Note that, when $r=3$, Lemma~\ref{lem:r} gives $2|S|-|V(CH(S))|-2$ pairwise interior-disjoint empty triangles. This agrees with the classical result that every triangulation of $S$ has precisely this many bounded triangular faces. For $r=4$, a point set $S$ admits a quadrangulation of $CH(S)$ if and only if $|V(CH(S))|$ is even
\cite[Theorem~3.1]{bose1997characterizing}.\footnote{A {\it polygonal partition} $\mathcal P(S)$ of a planar point set $S$ in general position is a partition of $CH(S)$ into simple polygons (faces) such that the vertex set of $\mathcal P (S)$ is $S$ and each edge of $\mathcal P(S)$ which is not an edge of $CH(S)$ is common to exactly two faces.  A polygonal partition $\mathcal P(S)$ of $S$ is said to be an $r$-{\it angulation} if  every face of $\mathcal P(S)$ is an $r$-gon (not necessarily convex).}  
If $|V(CH(S))|$ is odd, their construction instead gives a partition of $CH(S)$ into quadrilaterals and a single triangle.
In either case, the number of quadrilaterals is $|S|-\lceil \frac{|V(CH(S))|}{2} \rceil-1$. This agrees with the bound in \eqref{eq:rpoints} when $|S|$ is even or $|V(CH(S))|$ is odd, and exceeds it by one otherwise. More generally, Euler's formula and angle counting implies that an $r$-angulation
of $CH(S)$ with vertex set $S$, whenever one exists, has exactly
$$\frac{2|S|-|V(CH(S))|-2}{r-2}$$ bounded faces, in particular, this quotient must be an integer (see, for example, \cite[Section 3.2]{aichholzer2008matching}). 
Moreover,
\begin{equation}\label{eq:fc}
\left\lfloor
\frac{2|S|-|V(CH(S))|-2}{r-2}
\right\rfloor
\end{equation}
is an upper bound on the number of pairwise interior-disjoint
empty $r$-gons determined by $S$. The upper bound \eqref{eq:fc} exceeds
the lower bound in \eqref{eq:rpoints} by at most one. Related existence results for $r$-angulations with a possibly nonconvex outer boundary are obtained in \cite{payne2014angulation}, however, these do not directly establish the attainability of \eqref{eq:fc}. For $r\geq5$, an issue that requires care is that grouping $r-2$ adjacent triangles in an arbitrary manner may surround an interior vertex or otherwise fail to be a simple $r$-gon. The spiral serpentine polygonization of \citet{iwerks2011spiral} avoids this difficulty and gives the number of $r$-gons in Lemma~\ref{lem:r}. It would be interesting to determine whether, for each fixed $r\geq5$ and all sufficiently large point sets, the bound in \eqref{eq:rpoints} can be improved to \eqref{eq:fc}. For our applications, however, an additive difference of at most one is immaterial, and \eqref{eq:rpoints} suffices.  
\end{remark}

We now proceed with the proof of Theorem \ref{thm:c2}. Fix an integer $r\geq 3$ and for notational convenience denote 
$$ s:=\left\lfloor \frac{r-2}{2}\right\rfloor $$
and $h_1:=4r^2$. We will show that $\mathsf{M}_r^\star(2,s)\leq 2\mathsf{ES}(h_1)$.
Suppose, for a contradiction, that there exists a set $S$ of at least 
$2\mathsf{ES}(h_1)$ points in general position, colored with two colors, such that every monochromatic $r$-gon contains at least $s+1$ interior points. One of the two color classes has size at least $\mathsf{ES}(h_1)$. Hence, by the Erd\H{o}s--Szekeres theorem, $S$ contains a monochromatic convex $h_1$-gon. Among all monochromatic convex $h_1$-gons in $S$, choose one, say $T$, with the minimum number of points of $S$ in its interior. Without loss of generality, suppose the color of $T$ is $1$.
Define
$$ T_1:=CH(T)\cap S_1, $$ 
where $S_1$ is the set of points with color 1. Note that $|V(CH(T_1))| = h_1$ and $|T_1| = h_1+i_1$, where $i_1$ is the number of points of color 1 inside $T$. Hence, by Lemma \ref{lem:r}, $T_1$ contains at least
$$\left\lfloor \frac{|T_1|-2}{r-2}\right\rfloor+ \left\lfloor \frac{|T_1|-|V(CH(T_1))|}{r-2}\right\rfloor = \left\lfloor \frac{h_1+i_1-2}{r-2}\right\rfloor+ \left\lfloor \frac{i_1}{r-2}\right\rfloor $$
pairwise interior-disjoint $r$-gons with vertices of color $1$ which have no points of color 1 in their interiors. Moreover, by assumption, each of these $r$-gons contains at least $s + 1 $ interior points (all of which have color 2). Hence, $CH(T_1)$ has at least
$$(s+1)\Bigg(\left\lfloor \frac{h_1+i_1-2}{r-2}\right\rfloor+ \left\lfloor \frac{i_1}{r-2}\right\rfloor \Bigg)\geq \frac{(s+1)(2i_1+h_1-2r+2)}{r-2}. $$
interior points with color 2. Let $T_2:=CH(T_1)\cap S_2=CH(T)\cap S_2$. Note that  
\begin{align}\label{eq:t2c2}
t_2:=|T_2|\geq
\frac{(s+1)(2i_1+h_1-2r+2)}{r-2}.
\end{align}
By the minimality of $T$, $|V(CH(T_2))|<h_1$. Hence, by Lemma \ref{lem:r}, $T_2$ contains at least
$$\left\lfloor \frac{t_2-2}{r-2}\right\rfloor+\left\lfloor \frac{t_2-|V(CH(T_2))|}{r-2}\right\rfloor  \geq  \left\lfloor \frac{t_2-2}{r-2}\right\rfloor+ \left\lfloor \frac{t_2-h_1+1}{r-2}\right\rfloor  $$
pairwise interior-disjoint empty $r$-gons of color $2$ which have no points of color 2 in their interiors. By assumption each of these color-$2$ $r$-gons contains at least $s+1$ interior points, all of color $1$.  Therefore,  
\begin{align}  
 |CH(T_2) \cap S_1| & \geq (s + 1 )\Bigg(\left\lfloor \frac{t_2-2}{r-2}\right\rfloor+ \left\lfloor \frac{t_2-h_1+1}{r-2}\right\rfloor \Bigg) \nonumber \\ 
& \geq (s + 1 )\frac{(2 t_2 - h_1 - 2r + 2)}{r-2}  \nonumber \\
& \geq \frac{s+1}{r-2} \left[ \frac{2(s+1)(2i_1+h_1-2r+2)}{r-2} -h_1-2r+2 \right]  ,  
\label{eq:t2r}  
\end{align} 
where the last step follows from \eqref{eq:t2c2}.

Now, we consider the following 2 cases, depending on $r$ is even or odd. 
\begin{itemize} 
\item $r=2 \ell$ is even, for $\ell \geq 2$. Then $s=\ell-1$ and \eqref{eq:t2r} becomes,  
\begin{align*}
|CH(T_2)\cap S_1|
&\ge
\frac{\ell^2}{(\ell-1)^2}\,i_1
+
\frac{\ell}{2(\ell-1)^2}\,h_1
-
\frac{\ell(4\ell^2-4\ell+1)}{(\ell-1)^2} \nonumber \\
& \geq i_1+ 1 ,
\end{align*}  
since $\frac{\ell^2}{(\ell-1)^2}\geq 1$ and $h_1=4r^2=16\ell^2$ gives $\frac{\ell}{2(\ell-1)^2}h_1-
\frac{\ell(4\ell^2-4\ell+1)}{(\ell-1)^2}
\geq 1$ for $\ell \geq 2$. This is a contradiction, since $|CH(T_2) \cap S_1| \leq i_1$.  

\item $r=2\ell+1$ is odd, where $\ell\geq 1$. Then $s=\ell-1$  and \eqref{eq:t2r} becomes 
\begin{align*}
|CH(T_2)\cap S_1| &\ge \frac{4\ell^2}{(2\ell-1)^2}\,i_1 + \frac{\ell}{(2\ell-1)^2}\,h_1 -
\frac{ 4 \ell^2(4 \ell -1 )}{(2\ell-1)^2}.
\end{align*}
Since $\frac{4\ell^2}{(2\ell-1)^2}\geq1$ and $h_1=4(2\ell+1)^2$, the RHS above is at least $i_1+1$, which contradicts $|CH(T_2)\cap S_1|\leq i_1$.
\end{itemize}

Since both the above cases lead to a contradiction, it follows that any set $S$, with $|S| \geq 2 \mathsf{ES}(h_1)$, colored with 2 colors contains a monochromatic $r$-gon with at most $s = \left \lfloor \frac{r-2}{2} \right \rfloor$ interior points. This implies that $\mathsf{M}_r^{\star}(2, s) \leq 2 \mathsf{ES}(h_1)$ and, hence, $\lambda_r^\star(2) \leq s$. This completes the proof of Theorem \ref{thm:c2}. \hfill $\Box$

\section{Proof of Theorem \ref{thm:hexagon}} 
\label{sec:hexagonpf}

Let $S$ be a set of points in the plane in general position colored with two colors, which we will refer to as 1 and 2.  Denote by $S_1$ and $S_2$ the subsets of points $S$ with color 1 and color 2, respectively. Without loss of generality, assume $|S_1| \geq |S_2|$.  Define 
\begin{align}\label{eq:convexpoints}
S_2' = CH(S_1) \cap S_2  \quad \text{ and } \quad S_1' = CH(S_2') \cap S_1  .   
\end{align}  
Here, $S_2'$ is the set of points of color 2 lying in the interior of $CH(S_1)$ and $S_1'$ is the set of points of color 1 lying in the interior of $CH(S_2')$ (see Figure \ref{fig:S}).  The main ingredients of the proof of Theorem \ref{thm:hexagon} are the following two propositions.

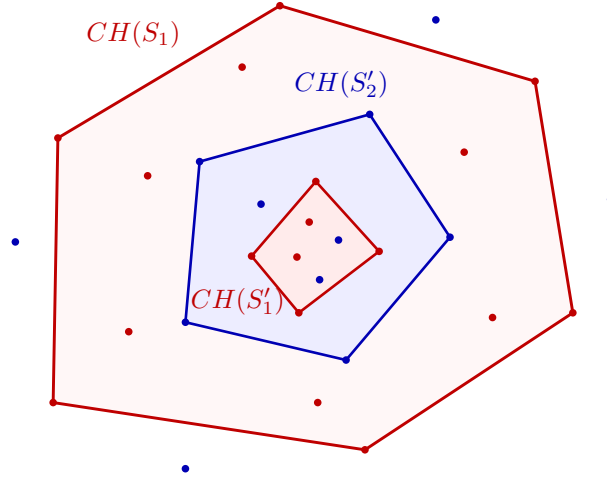
\begin{figure}[htbp]
\centering
\begin{tikzpicture}[x=1.25cm,y=1.25cm,font=\small]
\colorlet{hullred}{red!75!black}
\colorlet{hullblue}{blue!70!black}

\filldraw[fill=red!3,draw=hullred,line width=1.1pt]
  (-2.70,-1.50)--(.60,-2.00)--(2.80,-.55)
  --(2.40,1.90)--(-.30,2.70)--(-2.65,1.30)--cycle;

\filldraw[fill=blue!7,draw=hullblue,line width=1pt]
  (-1.30,-.65)--(.40,-1.05)--(1.50,.25)
  --(.65,1.55)--(-1.15,1.05)--cycle;

\filldraw[fill=red!8,draw=hullred,line width=1pt]
  (-.60,.05)--(-.10,-.55)--(.75,.10)--(.08,.84)--cycle;

\foreach \x/\y in {
  -2.70/-1.50,.60/-2.00,2.80/-.55,
  2.40/1.90,-.30/2.70,-2.65/1.30,
  -1.90/-.75,-1.70/.90,-.70/2.05,
  1.65/1.15,1.95/-.60,.10/-1.50}
  \fill[hullred] (\x,\y) circle[radius=.04];

\foreach \x/\y in {
  -.60/.05,.08/.84,.75/.10,-.10/-.55,-.12/.04,.01/.41}{
  \fill[hullred] (\x,\y) circle[radius=.04];
}

\foreach \x/\y in {
  -1.30/-.65,.40/-1.05,1.50/.25,.65/1.55,-1.15/1.05,
  -.50/.60,.32/.22,.12/-.20}{
  \fill[hullblue] (\x,\y) circle[radius=.04];
}

\foreach \x/\y in {-3.10/.20,-1.30/-2.20,3.20/.65,1.35/2.55}
  \fill[hullblue] (\x,\y) circle[radius=.04];

\node[text=hullred] at (-1.85,2.40)
  {$CH(S_1)$};
\node[text=hullblue] at (.35,1.87)
  {$CH(S_2')$};
\node[text=hullred] (innerlabel) at (-.745, -0.425)
  {$CH(S_1')$};

%
%
%
%
\end{tikzpicture}
\caption{ \small{Construction of $S_2'$ and $S_1'$ (recall \eqref{eq:convexpoints}).  }}  
\label{fig:S}  
\end{figure}

\begin{proposition}\label{prop:lhull}
If, for some integer  $t \geq 2$, 
$$|V(CH(S_1))| \geq \frac{2}{t+1} |S_1| + 10 (t-1),$$
then $S$ contains a monochromatic hexagon with at most one interior point.
\end{proposition}

\begin{proposition}\label{prop:shull}
Suppose $|S_1| \geq 655$ and 
\begin{equation}
|V(CH(S_1))| < \frac{|S_1| - 54}{44}  .   
\label{eq:shull}
\end{equation}
Then $S$ contains a monochromatic hexagon with at most one interior point.
\end{proposition}

The proofs of the above propositions are given in Section \ref{sec:lhullpf} and Section \ref{sec:shullpf}. To complete the proof of Theorem \ref{thm:hexagon}, using these two results, suppose $S$ is a 2-colored point set with $|S|\geq 306447$. Relabel the colors, if necessary, so that $|S_1|\geq |S_2|$. Hence, $|S_1| \geq 153224$. Suppose, for a contradiction, that $S$ does not contain a monochromatic hexagon with at most one interior point. Applying Proposition~\ref{prop:lhull} with $t=174$, we have
\begin{equation} \label{eq:hull_bound}
|V(CH(S_1))| < \frac{2}{175} |S_1| + 1730.
\end{equation}
Further, since $|S_1| \geq 153224 > 655$, by Proposition \ref{prop:shull} and \eqref{eq:hull_bound},
\[
|S_1| \leq 44 |V(CH(S_1))| + 54 < \frac{88}{175} |S_1| + 76174  ,  
\]
which implies
\[
|S_1| < \frac{13330450}{87} < 153224,
\]
which contradicts $|S|\geq 306447$.   \hfill $\Box$

\subsection{ Proof of Proposition \ref{prop:lhull} }
\label{sec:lhullpf}

Let $S_1', S_2'$ be as defined in \eqref{eq:convexpoints}. We first show that if $|S_2'|$ is small, then $S$ contains a monochromatic hexagon with at most one interior point.

\begin{lemma} \label{lem:hull} 
If $|S_2'| \leq \frac{1}{2} (|S_1| + |S_1'| - 10)$, then there exists a hexagon of color $1$ with at most one interior point.
\end{lemma}  

\begin{proof} Applying Lemma \ref{lem:r} to $S_1$, we have at least 
    \[
    \left\lfloor \frac{|S_1|-2}{4} \right \rfloor + \left \lfloor \frac{|S_1|-|V(CH(S_1))|}{4}  \right\rfloor
    \]
    pairwise interior-disjoint empty hexagons of color 1. To prevent hexagons of color 1 with at most one interior point, we must have that
    \begin{align*}
        |S_2'| & \geq 2 \left(\left\lfloor \frac{|S_1|-2}{4} \right \rfloor + \left \lfloor \frac{|S_1|-|V(CH(S_1))|}{4}  \right\rfloor \right) \\
        & > 2 \left(\frac{|S_1|-2}{4} - 1 + \frac{|S_1|-|V(CH(S_1))|}{4} - 1 \right) \\
        & \geq \frac{1}{2} \left(|S_1| + |S_1'| - 10 \right). 
    \end{align*} 
    This completes the proof of Lemma \ref{lem:hull}. 
\end{proof}

We now iterate the nested hull construction as follows: Begin with $P_0 = S_1$. For $i \geq 0$, define 
$$P_{i+1} \coloneqq
    \begin{cases}
        CH(P_i) \cap S_2, & \text{if } i \text{ is even}, \\
        CH(P_i) \cap S_1, & \text{if } i \text{ is odd}.
    \end{cases}
$$
If $S$ does not contain any monochromatic hexagon with at most one interior point, then by Lemma \ref{lem:hull}, 
\begin{align*}
    |P_{i+1}| > \frac{1}{2} \left(|P_i|+|P_{i+2}| - 10 \right)  ,   
\end{align*} 
for all  $i \geq 0$. Rearranging the terms, $|P_{i+1}| - |P_{i+2}| > |P_i| - |P_{i+1}| - 10$. Repeatedly applying this inequality for $t$ steps gives, for $t \geq 1$, 
 \begin{align} \label{eq:recurrence2}
        |P_{i+t}| - |P_{i+t+1}| > |P_i| - |P_{i+1}| - 10 t  .  
    \end{align}
    Hence, 
    \begin{align*}
    |P_{i}| - |P_{i+k+1}| & = \sum_{t=0}^{k} \left(|P_{i+t}| - |P_{i+t+1}| \right) \\
    & > |P_i| - |P_{i+1}| + \sum_{t=1}^{k} \left(|P_{i}| - |P_{i+1}| - 10 t \right) \tag*{(from \eqref{eq:recurrence2})} \\
    & = (k+1) (|P_i| - |P_{i+1}|) - 5 k (k+1) . 
    \end{align*}
    This implies, since $|P_{i+k+1}| \geq 0$, $|P_{i+1}| > \frac{k}{k+1} |P_i| - 5 k$. 
    Applying this with $i=0$ and $k=t$ gives, 
\begin{align} \label{eq:disc_type1}
    |P_1| > \frac{t}{t+1} |P_0| - 5 t
\end{align}
Also, applying the same inequality with $i=1$ and $k=t-1$, 
\begin{align*}
    |P_2|  > \frac{t-1}{t} |P_1| - 5 (t-1)  & > \frac{t-1}{t} \left(\frac{t}{t+1} |P_0| - 5 t \right) - 5 (t-1) \\
    & = \frac{t-1}{t+1} |P_0| - 10 (t-1)  ,  
\end{align*}
where the second inequality uses \eqref{eq:disc_type1}. Hence, 
\begin{align*}
    |V(CH(S_1))| = |V(CH(P_0))|  \leq |P_0| - |P_2| & < \frac{2}{t+1} |P_0| + 10 (t-1) \\
    & = \frac{2}{t+1} |S_1| + 10 (t-1),
\end{align*}
as required. This completes the proof of Proposition \ref{prop:lhull}.  \hfill  $\Box$

\subsection{Proof of Proposition \ref{prop:shull} }
\label{sec:shullpf}

We begin with the notion of 4-connected triangulations of planar point sets:

\begin{definition}[{\cite{DGR}}]
Given a finite point set $P\subset\mathbb{R}^{2}$ in general position, a \emph{$4$-connected triangulation} of $P$ is a plane graph $T$ with vertex set $P$ whose edges are line segments, whose bounded faces are triangles, and whose outer-face boundary is the  boundary of $CH(P)$, such that $T$ is $4$-connected, that is, $T$ has at least five vertices and remains connected after deleting any set of at most three vertices.
\end{definition}


\citet{DGR} provide a characterization of point sets admitting a 4-connected triangulation. Using this characterization, we show that the larger color class of any sufficiently large bichromatic point set admits a $4$-connected triangulation, provided that the point set contains no monochromatic hexagon with
at most one interior point. Throughout, for a contradiction, we will assume that $S$ is a 2-colored point set, satisfying the hypothesis of Proposition \ref{prop:shull}, where every monochromatic hexagon has at least 2 interior points.

\begin{lemma}\label{lem:S}
Suppose $|S_1|\geq655$ and $S$ does not contain a monochromatic hexagon with at most one interior point. Then $S_1$ admits a straight-line $4$-connected triangulation.
\end{lemma}

\begin{proof}
Applying Proposition~\ref{prop:lhull} to $(S_1,S_2)$ with $t=16$ gives,  
$$|V(CH(S_1))|<\frac{2}{17}|S_1|+150.$$
Let $P:=S_1\setminus V(CH(S_1))$ be the set of color $1$ points strictly inside $CH(S_1)$.  
Consider the restricted bichromatic point set $S\cap CH(P)$,
whose color $1$ class is precisely $P$. Applying Proposition~\ref{prop:lhull} to $S\cap CH(P)$ with $t=16$ gives, 
$$|V(CH(P))| <\frac{2}{17}|P|+150 =\frac{2}{17}\bigl(|S_1|-|V(CH(S_1))|\bigr)+150.$$
Combining the above, 
\begin{align*}
2|V(CH(S_1))|+|V(CH(P))| <\frac{32}{17}|V(CH(S_1))| +\frac{2}{17}|S_1|+150 <\frac{98}{289}|S_1|+\frac{7350}{17} < |S_1|,
\end{align*}
where the last inequality follows from $|S_1|\geq 655$. Consequently,
\begin{align*}
|P|-|V(CH(P))| =|S_1|-|V(CH(S_1))|-|V(CH(P))| >|V(CH(S_1))|.
\end{align*}
This inequality implies Necessary Condition~3 of~\cite{DGR}, whose RHS is $|V(CH(S_1))|$ minus a nonnegative quantity. By the characterization in \cite{DGR}, this condition is also sufficient. Hence, $S_1$ admits a straight-line $4$-connected triangulation.  
\end{proof}  

%
%

\subsubsection{Properties of 4-Connected Triangulations }

Let $ T:= T(S_1)$ be a 4-connected triangulation of $S_1$, which exists because of Lemma \ref{lem:S}, since $|S_1|  \geq 655$. Denote by $G_{T(S_1)}$ the dual of $T(S_1)$, where the vertices are the bounded triangular faces of $T$ and two triangles are adjacent in the dual if they share an interior edge. The following observation lists some elementary properties of $G:=G_{T(S_1)}$.

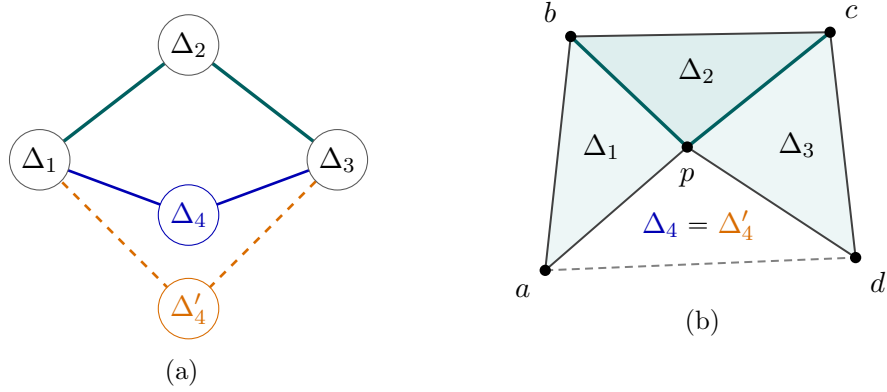
\begin{figure}[tbp]
  \centering
\begin{tikzpicture}[
  x=1.125cm,y=1.125cm,
  font=\normalsize,
  shared/.style={draw=teal!75!black,line width=1.3pt},
  first/.style={draw=blue!70!black,line width=1pt},
  second/.style={draw=orange!85!black,line width=1pt,dashed},
  dual/.style={circle,draw=black!65,fill=white,
               minimum size=8mm,inner sep=1pt},
  primal/.style={circle,fill=black,inner sep=1.5pt},
  rim/.style={draw=black!75,line width=.75pt}
]

\begin{scope}  
  \coordinate (u1) at (-1.75,0);
  \coordinate (u2) at (0,1.35);
  \coordinate (u3) at (1.75,0);
  \coordinate (u4) at (0,-.65);
  \coordinate (u4p) at (0,-1.75);

  \draw[shared] (u1)--(u2)--(u3);
  \draw[first] (u1)--(u4)--(u3);
  \draw[second] (u1)--(u4p)--(u3);

  \node[dual] at (u1) {$\Delta_1$};
  \node[dual] at (u2) {$\Delta_2$};
  \node[dual] at (u3) {$\Delta_3$};
  \node[dual,draw=blue!70!black,text=blue!70!black]
    at (u4) {$\Delta_4$};
  \node[dual,draw=orange!85!black,text=orange!85!black]
    at (u4p) {$\Delta'_4$};
\end{scope}

\begin{scope}[xshift=6.8cm] 
  \coordinate (a) at (-1.85,-1.30);
  \coordinate (b) at (-1.55,1.45);
  \coordinate (c) at (1.50,1.50);
  \coordinate (d) at (1.80,-1.15);
  \coordinate (p) at (-.18,.15);

  \fill[teal!7] (p)--(a)--(b)--cycle;
  \fill[teal!13] (p)--(b)--(c)--cycle;
  \fill[teal!7] (p)--(c)--(d)--cycle;
  \draw[rim] (a)--(b)--(c)--(d);
  \draw[rim] (a)--(p)--(d);
  \draw[shared] (b)--(p)--(c);
  \draw[gray,densely dashed,line width=.75pt] (a)--(d);

  \node at (-1.19,.15) {$\Delta_1$};
  \node at (-.07,1.04) {$\Delta_2$};
  \node at (1.11,.17) {$\Delta_3$};
  \node at (-.05,-.77)
    {$\textcolor{blue!70!black}{\Delta_4}
       =\textcolor{orange!85!black}{\Delta'_4}$};
  \node[font=\small] at (-6.125,-2.5)
    {(a)};
    
  \node[primal,label=below left:$a$] at (a) {};
  \node[primal,label=above left:$b$] at (b) {};
  \node[primal,label=above right:$c$] at (c) {};
  \node[primal,label=below right:$d$] at (d) {};
  \node[primal,label={[label distance=2pt]below:$p$}] at (p) {};
  \node[font=\small] at (0,-1.91)
    {(b)};
\end{scope} 
\end{tikzpicture}
\caption{\small{ Illustration for the proof of Observation \ref{obs:dual}: (a) The two $4$-cycles share the common path  $(\Delta_1, \Delta_2, \Delta_3)$ in the dual and (b) their primal configuration.  }}  
\label{fig:triangle}
\end{figure}

\begin{observation}\label{obs:dual} The graph $G$ has the following properties: 

\begin{itemize} 

\item[$(1)$] $G$ is triangle-free with maximum degree at most three.

\item[$(2)$] Every $4$-cycle of $G$ corresponds to four triangular faces incident with an interior vertex of $T$ of degree four.

\item[$(3)$] Two distinct $4$-cycles of $G$ cannot share a $2$-edge path.
\end{itemize}
\end{observation}

\begin{proof}
Every triangular face in $T:= T(S_1)$ has three sides. Hence, every vertex of $G$ has degree at most three. Now, suppose that $G$ contains a triangle. Then the graph formed by the boundaries of the corresponding three triangular faces in the primal is a plane 4-clique in $T$. This means $T$ has an interior vertex of degree 3, which contradicts the 4-connectedness of $T$. Hence, $G$ is triangle-free. This proves the result in (1). 

Now, suppose that $G$ contains a $4$-cycle $(\Delta_1,\Delta_2,\Delta_3,\Delta_4)$ (see
Figure~\ref{fig:triangle}(a)). Then the corresponding triangular faces in the primal form a wheel centered at an interior vertex $p$, as shown in Figure \ref{fig:triangle}(b). These are precisely the four faces incident with $p$, which has degree 4 in $T$.  This proves (2). 

To show (3), suppose that $G$ contains two $4$-cycles $(\Delta_1,\Delta_2,\Delta_3,\Delta_4)$ and $(\Delta_1,\Delta_2,\Delta_3,\Delta'_4)$ sharing the two-edge path
$(\Delta_1,\Delta_2,\Delta_3)$, as shown in Figure \ref{fig:triangle}(a). In the primal, the two shared adjacencies correspond to two sides of $\Delta_2$ meeting at a vertex $p$ (see
Figure~\ref{fig:triangle}(b)). By (2), both cycles correspond to the four triangular faces incident with $p$. Hence, their fourth faces must coincide, giving $\Delta_4=\Delta'_4$.   
\end{proof}

\subsubsection{4-Vertex Trees in the Dual are Hexagons in the Primal }

Throughout, $G:=G_{T(S_1)}$ will denote the dual of a 4-connected triangulation $T:=T(S_1)$ of $S_1$. For $A \subseteq V(G)$, $G[A]$ is the induced subgraph of $G$ on the vertex set $A$.

\begin{lemma}\label{lem:T} 
Suppose $A \subseteq V(G)$, with $|A|=4$ and $G[A]$ is a tree. Then the union of four primal triangular faces corresponding to $A$ is a simple polygon with six boundary vertices. In particular, it forms a hexagon of color 1 with no points of color 1 in the interior.  
\end{lemma}

\begin{figure}[h]  
\centering 
\begin{tikzpicture}[
  x=1.125cm,y=1.125cm, font=\small,
 treevertex/.style={circle,fill=red!65!black,inner sep=1.6pt},
 treedual/.style={circle,draw=teal!75!black,fill=white,
   minimum size=6mm,inner sep=1pt,font=\small},
 treeedge/.style={line width=.75pt,draw=black!75},
 treeboundary/.style={line width=1.1pt,draw=black!85},
 treedualedge/.style={teal!75!black,line width=.9pt},  
 treetitle/.style={font=\small\bfseries,align=center}
]

\begin{scope}
 \coordinate (q0) at (0,0);
 \coordinate (q1) at (0,-1.15);
 \coordinate (q2) at (1.35,.9);
 \coordinate (q3) at (-1.35,.9);
 \draw[treedualedge] (q1)--(q0)--(q2) (q0)--(q3);
 \foreach \i in {0,1,2,3}
   \node[treedual] at (q\i) {$\Delta_{\i}$};
    \node[font=\small] at (0.01, -2.35)
    {(a)};
\end{scope}

\begin{scope}[shift={(6.8,0)},treedual/.append style={minimum size=5mm,inner sep=.4pt,font=\footnotesize}] 
 \coordinate (a) at (-1,-.55);
 \coordinate (b) at (1,-.5);
 \coordinate (c) at (.05,1.1);
 \coordinate (x) at (-.05,-1.4);
 \coordinate (y) at (1.85,.6);
 \coordinate (z) at (-1.8,.75);
 \fill[teal!17] (a)--(b)--(c)--cycle;
 \fill[teal!7] (a)--(b)--(x)--cycle;
 \fill[teal!7] (b)--(c)--(y)--cycle;
 \fill[teal!7] (c)--(a)--(z)--cycle;
 \draw[treeboundary] (a)--(x)--(b)--(y)--(c)--(z)--cycle;
 \draw[treeedge] (a)--(b)--(c)--cycle;
 \coordinate (q0) at (barycentric cs:a=1,b=1,c=1);
 \coordinate (q1) at (barycentric cs:a=1,b=1,x=1);
 \coordinate (q2) at (barycentric cs:b=1,c=1,y=1);
 \coordinate (q3) at (barycentric cs:c=1,a=1,z=1);
 \draw[treedualedge] (q1)--(q0)--(q2) (q0)--(q3);
 \foreach \i in {0,1,2,3}
   \node[treedual] at (q\i) {$\Delta_{\i}$};
 \foreach \v/\pos in {a/below left,b/below right,c/above,
                       x/below,y/right,z/left}
   \node[treevertex,label=\pos:$\v$] at (\v) {};
    \node[font=\small] at (0.01,-2.35)
    {(b)};
\end{scope}
\end{tikzpicture} 
\caption{\small{Illustration for the proof of Lemma \ref{lem:T}: (a) Four vertices forming an induced 3-star in the dual and (b) their primal configuration.  }  } 
\label{fig:graphs}
\end{figure}
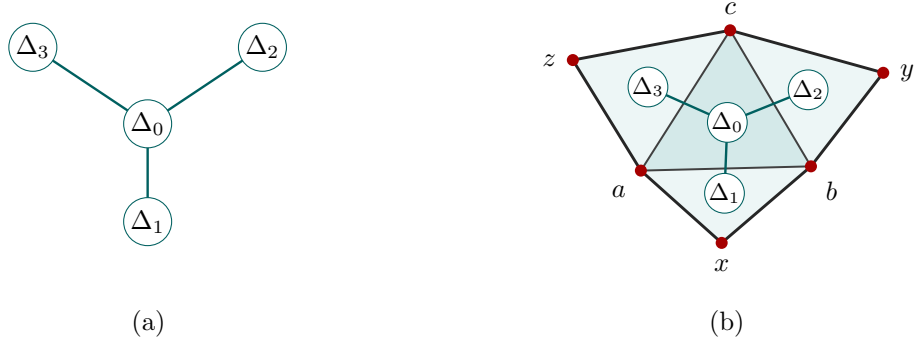

\begin{proof}
Since $G[A]$ is a tree on four vertices, it is isomorphic to either the three-star $K_{1,3}$ or the three-edge path $P_4$.

\begin{itemize} 

\item $G[A]\cong K_{1, 3}$: Suppose $G[A]$ is a 3-star with central vertex $\Delta_0$ and leaves $\Delta_1, \Delta_2, \Delta_3$ (as shown in Figure \ref{fig:graphs} (a)). Suppose the central vertex $\Delta_0$ corresponds to the triangular face $(a, b, c)$ and the leaf vertices  $\Delta_1, \Delta_2, \Delta_3$ correspond to $(a, b, x)$, $(b, c, y)$, and $(c, a, z)$, respectively, in the primal. Note that each of $x,y,z$ is different from $a,b,c$, since otherwise its face would repeat the triangle $(a, b, c)$. Moreover, $x,y,z$ are pairwise distinct:
identifying any two would make the corresponding leaves share an edge, contradicting the induced star condition. Thus, the three leaf triangles are attached along distinct sides of the triangle $(a, b, c)$, as shown in Figure \ref{fig:graphs} (b). The boundary of the union of these four triangles is the hexagon $(a,x,b,y,c,z)$ of color 1 points with no other color 1 in the interior.

\item $G[A]\cong P_4$:  Suppose the vertices of $G[A]$ in path order are  $\Delta_1,\Delta_2,\Delta_3,\Delta_4$ (as shown in Figure~\ref{fig:graphp}(a)). Relabeling the primal vertices, suppose $\Delta_1,\Delta_2,\Delta_3$ correspond to the triangular faces $(a, b, c)$, $(a, c, d)$, and $(a, d, e)$, respectively. Now there are two possibilities: 

\begin{itemize} 

\item The dual vertex $\Delta_4$ introduces an additional vertex $f$ in the primal. Then the resulting boundary is a hexagon of color 1 with no other color 1 points in the interior (two allowed configurations are shown in Figure~\ref{fig:graphp}(b) and Figure~\ref{fig:graphp}(c)).

\item The dual vertex $\Delta_4$ introduces no additional vertex in the primal. This creates a configuration as in Figure~\ref{fig:graphp}(d). Then the triangle $(a,b,d)$ separates $c$ from $e$. Hence, deleting the vertices $a$, $b$, and $d$ disconnects $T$, which contradicts the $4$-connectedness of $T$. 

\end{itemize}
\end{itemize} 
This completes the proof of Lemma \ref{lem:T}.  
\end{proof}

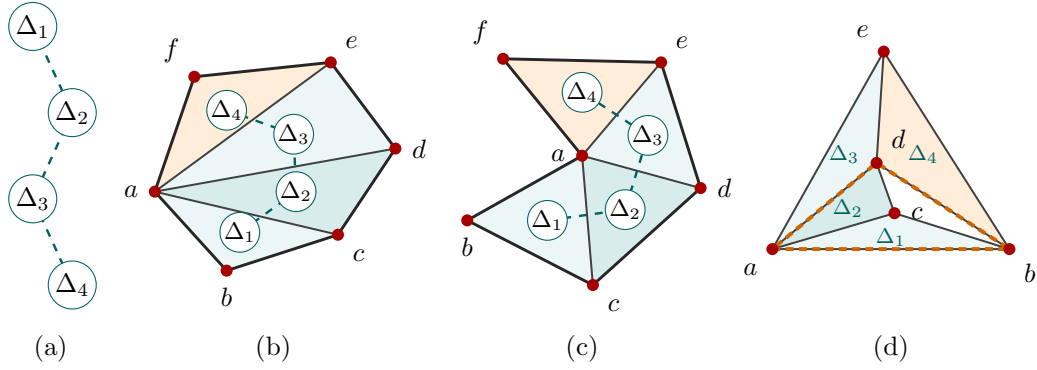
\begin{figure}[h]
\centering
\begin{tikzpicture}[
 x=0.95cm,y=0.95cm,font=\small,
 treevertex/.style={circle,fill=red!65!black,inner sep=1.6pt},
 treedual/.style={circle,draw=teal!75!black,fill=white,
   minimum size=6mm,inner sep=1pt,font=\small},
 treeedge/.style={line width=.75pt,draw=black!75},
 treeboundary/.style={line width=1.1pt,draw=black!85},
 treedualedge/.style={teal!75!black,dashed,line width=.9pt},
 panelcaption/.style={font=\small,align=center,anchor=north,
   text width=4cm}
]

\begin{scope}[rotate=-90]
 \coordinate (q1) at (-1.8,-.25);
 \coordinate (q2) at (-.6,.3);
 \coordinate (q3) at (.6,-.25);
 \coordinate (q4) at (1.8,.3);
 \draw[treedualedge] (q1)--(q2)--(q3)--(q4);
 \foreach \i in {1,2,3,4}
   \node[treedual] at (q\i) {$\Delta_{\i}$};
\end{scope}
\node[font=\small,anchor=north] at (0,-2.35) {(a)};

\begin{scope}[shift={(3.1,0)},
 treedual/.append style={minimum size=5mm,inner sep=.4pt,
                         font=\footnotesize}]
 \coordinate (a) at (-1.65,-.5);
 \coordinate (b) at (-.65,-1.6);
 \coordinate (c) at (.9,-1.1);
 \coordinate (d) at (1.7,.1);
 \coordinate (e) at (.8,1.3);
 \coordinate (f) at (-1.1,1.1);
 \fill[teal!7] (a)--(b)--(c)--cycle;
 \fill[teal!15] (a)--(c)--(d)--cycle;
 \fill[teal!7] (a)--(d)--(e)--cycle;
 \fill[orange!15] (a)--(e)--(f)--cycle;
 \draw[treeboundary] (a)--(b)--(c)--(d)--(e)--(f)--cycle;
 \draw[treeedge] (a)--(c) (a)--(d) (a)--(e);
 \coordinate (q1) at (barycentric cs:a=1,b=1,c=1);
 \coordinate (q2) at (barycentric cs:a=1,c=1,d=1);
 \coordinate (q3) at (barycentric cs:a=1,d=1,e=1);
 \coordinate (q4) at (barycentric cs:a=1,e=1,f=1);
 \draw[treedualedge] (q1)--(q2)--(q3)--(q4);
 \foreach \i in {1,2,3,4}
   \node[treedual] at (q\i) {$\Delta_{\i}$};
 \foreach \v/\pos in {a/left,b/below,c/below right,d/right,
                       e/above right,f/above left}
   \node[treevertex,label=\pos:$\v$] at (\v) {};
 \node[panelcaption] at (0,-2.35)
   { (b) };
\end{scope}

%
%

\begin{scope}[shift={(7.4,0)},
 treedual/.append style={minimum size=5mm,inner sep=.4pt,
                         font=\footnotesize}]
 \coordinate (a) at (0,0);
 \coordinate (b) at (-1.6,-.9);
 \coordinate (c) at (.15,-1.8);
 \coordinate (d) at (1.65,-.45);
 \coordinate (e) at (1.1,1.3);
 \coordinate (f) at (-1.1,1.35);
 \fill[teal!7] (a)--(b)--(c)--cycle;
 \fill[teal!15] (a)--(c)--(d)--cycle;
 \fill[teal!7] (a)--(d)--(e)--cycle;
 \fill[orange!15] (a)--(e)--(f)--cycle;
 \draw[treeboundary] (a)--(b)--(c)--(d)--(e)--(f)--cycle;
 \draw[treeedge] (a)--(c) (a)--(d) (a)--(e);
 \coordinate (q1) at (barycentric cs:a=1,b=1,c=1);
 \coordinate (q2) at (barycentric cs:a=1,c=1,d=1);
 \coordinate (q3) at (barycentric cs:a=1,d=1,e=1);
 \coordinate (q4) at (barycentric cs:a=1,e=1,f=1);
 \draw[treedualedge] (q1)--(q2)--(q3)--(q4);
 \foreach \i in {1,2,3,4}
   \node[treedual] at (q\i) {$\Delta_{\i}$};
 \foreach \v/\pos in {a/left,b/below,c/below right,d/right,
                       e/above right,f/above left}
   \node[treevertex,label=\pos:$\v$] at (\v) {};
 \node[panelcaption] at (0,-2.35)
   { (c) };
\end{scope}

\begin{scope}[shift={(11.7,0)}]
 \coordinate (a) at (-1.65,-1.3);
 \coordinate (b) at (1.65,-1.3);
 \coordinate (c) at (.05,-.80);
 \coordinate (d) at (-.2,-.1);
 \coordinate (e) at (-.1,1.45);
 \fill[teal!7] (a)--(b)--(c)--cycle;
 \fill[teal!15] (a)--(c)--(d)--cycle;
 \fill[teal!7] (a)--(d)--(e)--cycle;
 \fill[orange!15] (b)--(d)--(e)--cycle;
 \draw[treeedge] (a)--(b)--(c)--cycle (a)--(d)--(c)
   (a)--(e)--(d)--(b)--(e);
 \draw[orange!80!black,line width=1.5pt,densely dashed]
   (a)--(b)--(d)--cycle;
 \node[font=\scriptsize,text=teal!75!black] at (0,-1.14) {$\Delta_1$};
 \node[font=\scriptsize,text=teal!75!black] at (-.60,-.73) {$\Delta_2$};
 \node[font=\scriptsize,text=teal!75!black] at (-.65,.02) {$\Delta_3$};
 \node[font=\scriptsize,text=teal!75!black] at (.45,.02) {$\Delta_4$};
 \foreach \v/\pos in {a/below left,b/below right,c/right,d/above right,
                       e/above left}
   \node[treevertex,label=\pos:$\v$] at (\v) {};
 \node[panelcaption] at (0,-2.35)
   { (d) };
\end{scope}
\end{tikzpicture}  
\caption{\small{Illustration for the proof of Lemma \ref{lem:T}: (a) Four vertices forming an induced 3-edge path in the dual, (b)-(c) allowed primal configurations, and (d) forbidden configuration. }}    
\label{fig:graphp}
\end{figure}

\subsubsection{  Weighting Dual Vertices by the Number of Interior Points  }

Let $\lambda(T)$ be the number of bounded triangular faces in $T$, $e(T)$ the number of edges in $T$, and $\delta(T) = |V(CH(S_1))|$ the number of boundary vertices in $T$. Then counting edge-face incidences gives $3 \lambda(T) =2 e(T) - \delta(T)$ and applying Euler's formula gives $|S_1| - e(T)+\lambda(T)=1$. Combining these gives 
\begin{align}\label{eq:edge}
\lambda(T)=2 |S_1|- \delta(T)-2  .  
\end{align}  
For every bounded face $f$ of $T$, denote by $x_f$ the corresponding vertex in the dual $G$ and $d_{G}(x_f)$ the degree of $x_f$. Then $3-d_{G}(x_f)$ is exactly the number of its sides lying on the convex hull of $S_1$. Since every hull edge is incident with exactly one bounded face, 
\begin{equation}\label{eq:Euler}
\qquad \sum_{x \in V(G)}(3-d_{G}(x))= \delta(T).
\end{equation}
For the dual vertex $x_f$ corresponding to a bounded face $f$ of $T$, define the weight of the face $f$ and the total weight as follows:  
$$\mathrm{wt}(x_f)=|S_2 \cap\interior(f)| \quad \text{ and } \quad W=\sum_{x\in V(G)}\mathrm{wt}(x).$$
Note that $\mathrm{wt}(x_f)$ is the number of points of color 2 inside the face $f$ (which is a triangle of color 1) and $W$ is the sum of the weights over all bounded faces of $T$, which is the number of points of color 2 inside $CH(S_1)$. Hence, 
\begin{equation}\label{eq:W}
W =| S_2 \cap \interior(CH(S_1))|\le |S_2|\le |S_1|.
\end{equation}
Then Lemma \ref{lem:T} and the assumption that every monochromatic hexagon in $S$ contains at least 2 interior points imply, 
\begin{equation}\label{eq:fourT}
 \sum_{x \in A} \mathrm{wt}(x) \ge 2  ,  \quad\text{ if } |A|=4 \text{ and } G[A] \text{ is a tree}.
\end{equation}
Each vertex of $G$ with weight zero corresponds to an empty triangle of color 1 in the primal $T$. Denote the set of {\it zero vertices} by 
\begin{align}\label{eq:Z0}
Z_0=\{  x \in V(G): \mathrm{wt}(x)=0\}  .  
\end{align}
A \emph{zero component} of $G$ is a connected component of the induced graph $G[Z_0]$.

\subsubsection{Structure of Zero Components}

The following observation records some properties of a zero component. Let $N_G(x)$ be the set of neighbors of a vertex $x$ in $G$.

\begin{observation}\label{obs:z}
Let $F$ be a connected component of $G[Z_0]$. Then:
\begin{enumerate}
\item[$(1)$] $F$ is a single vertex, a single edge, a $2$-edge path, or a
$4$-cycle.

\item[$(2)$] If $F$ is a $4$-cycle, then every vertex $x \in V(G)\setminus V(F)$ satisfies $|N_G(x)\cap V(F)|\le 1$. Further, if $N_G(x)\cap V(F)\ne\varnothing$, then $\mathrm{wt}(x) \ge 2$.
\end{enumerate}
\end{observation}

\begin{proof}
A vertex of $G[Z_0]$ with three neighbors in $G[Z_0]$ would, together with those neighbors, induce a 3-star, since $G$ is triangle-free. The total weight of this 3-star would be zero, contradicting \eqref{eq:fourT}. Thus, every vertex of $G[Z_0]$ has degree at most 2, hence, any component $F$ is a path or a cycle. A path on at least four vertices, or a cycle of length
at least five, contains an induced 4-vertex path, again contradicting \eqref{eq:fourT}. Since $G$ is triangle-free, this proves (1).

For (2), suppose that $F$ is a $4$-cycle and that $x \notin V(F)$ has two neighbors on $F$. These neighbors cannot be consecutive vertices of $F$, since that would create a triangle in $G$. If they are opposite vertices of $F$, these two vertices and $x$ form a second $4$-cycle with one of their common neighbors on $F$. This cycle shares
a two-edge path with $F$, contradicting Observation~\ref{obs:dual} (3).  Hence $x$ can have at most one neighbor on $F$. Further, if $x$ is adjacent to $a\in V(F)$, then $x$, $a$, and the two neighbors of $a$ on $F$ induce a 3-star. The three vertices on $F$ have weight zero. Therefore, \eqref{eq:fourT} gives $\operatorname{wt}(x)\ge 2$.
\end{proof}

The next lemma describes the zero-weight neighbors of weight-one vertices.

\begin{lemma}\label{lem:unit}
Let $x \in V(G)$ satisfy $\operatorname{wt}(x)=1$. Then the set $N_G(x)\cap Z_0$ has exactly one of the following forms:
\begin{enumerate}
\item[$(1)$] At most two vertices, each forming a singleton component of $G[Z_0]$.

\item[$(2)$] A single vertex belonging to a single-edge component of $G[Z_0]$.

\item[$(3)$] The two endpoints of a $2$-edge-path component of $G[Z_0]$.
\end{enumerate}
Further, for each $2$-edge-path component $(a,b,c)$ of $G[Z_0]$, there is at most one vertex of weight one adjacent to both $a$ and $c$.
\end{lemma}

\begin{proof}
By Observation~\ref{obs:z} (2), $x$ has no neighbor in a $4$-cycle component of $G[Z_0]$. Now, we consider the following cases: 

\begin{itemize} 

\item First, suppose that $a \in N_G(x)\cap Z_0$, where $(a, b)$ is a single-edge component of $G[Z_0]$. Since $G$ is triangle-free, $x$ is not adjacent to $b$. If $x$ had another
zero-weight neighbor $d$, then $d$ would belong to a different component of $G[Z_0]$. Then $(b,a,x,d)$ would be an induced 4-vertex path of total weight one, contradicting
\eqref{eq:fourT}. This proves $(2)$. 

\item Next, suppose that $x$ has a neighbor in a $2$-edge-path component $(a,b,c)$ of $G[Z_0]$. If $x$ were adjacent to $b$, then triangle-freeness of $G$ would prevent adjacency to $a$ and $c$. Then $\{a,b,c,x\}$ would induce a 3-star of total weight one, contradicting \eqref{eq:fourT}. Thus, $x$ is adjacent to either $a$ or $c$. Note that if $x$ is adjacent to $a$, then it must also be adjacent to $c$ (and vice versa), otherwise $(x,a,b,c)$ would be an induced 4-vertex path of total weight one. Finally, any further zero-weight neighbor $d$ would lie in a different zero component and make $(b,a,x,d)$ an induced 4-vertex path of total weight one. This proves $(3)$. 

\item In the remaining case, every zero-weight neighbor of $x$ forms a singleton component of $G[Z_0]$. If there are three such neighbors, then these neighbors and $x$ combined would induce a 3-star of total weight one, again contradicting \eqref{eq:fourT}. Hence, $x$ can have at most two singleton zero-weight neighbors. This proves (1). 
\end{itemize}

Finally, if two distinct weight-one vertices $x, x' \in V(G)$ were adjacent to both endpoints of the same $2$-edge-path component $(a,b,c)$ of $G[Z_0]$, then $(a,b,c,x)$ and $(a,b,c,x')$ would be distinct $4$-cycles sharing the two-edge path $(a,b,c)$. This contradicts
Observation~\ref{obs:dual} (3).  
\end{proof}

\subsubsection{Incidence Counting}

Let $a_1, a_2, a_3, c_4$ be the numbers of zero components of $G$ that are isomorphic to a single vertex, a single edge, a $2$-edge path, and a $4$-cycle, respectively. Also, let $e_{\circ}$ and $e_{\bullet}$ denote the numbers of edges joining weight-one vertices in $G$ to zero components consisting of a single vertex and a single edge, respectively. Further, let $\gamma$ be the number of 2-edge-path zero components adjacent to a weight-one vertex.

\begin{lemma} With the above definitions, the following holds: 
\begin{align}\label{eq:points}
 3 |S_1| \le16 \delta(T) +18+8 c_4-15 \overline W - 6 a_2 , 
 \end{align}
where $\overline W :=\sum_{\mathrm{wt}(x)\geq 2} (\mathrm{wt}(x)-2)$. 
\end{lemma}

\begin{proof}
The following observations are immediate from the respective definitions: 
\begin{itemize}

\item There are exactly $e_{\circ}+e_{\bullet}+2 \gamma$ edges joining $Z_0$ to weight-one vertices. This is because, by Lemma~\ref{lem:unit}, each  2-edge-path zero component counted in $\gamma$ has exactly one weight-one neighbor, adjacent to both endpoints. 

\item There are at most $3b_{\geq2}$ edges from $Z_0$ to vertices of weight at least two, where $b_{\geq 2} := |\{v \in V(G): \mathrm{wt}(v) \geq 2\}|$ is the number of vertices of weight at least two in $G$, since the maximum degree of $G$ is 3. 

\end{itemize}  
Now, let $e_1$ denote the number of edges with exactly one endpoint in $Z_0$. Adding the degrees over $Z_0$, note that each edge with both end points in $Z_0$ is counted twice and exactly one end point in $Z_0$ is counted once. Hence, $$\sum_{x\in Z_0} d_G(x)=2|E(G[Z_0])|+e_1 .$$
This implies, 
$$3|Z_0|-2|E(G[Z_0])| = e_1 +\sum_{x\in Z_0}(3-d_G(x)).$$
Every edge leaving $Z_0$ leads to a positive-weight vertex. Exactly $e_{\circ}+e_{\bullet}+2\gamma$ such edges lead to weight-one vertices, while at most $3b_{\geq2}$ lead to vertices of weight at least two, since each has degree at
most three. Hence, $e_1 \leq e_{\circ}+e_{\bullet}+2\gamma+3b_{\geq2}$. Moreover, all degree deficits are nonnegative, hence, \eqref{eq:Euler} gives, 
$$ \sum_{x\in Z_0}(3-d_G(x)) \leq\sum_{x\in V(G)}(3-d_G(x)) = \delta(T).$$
Combining these bounds and using $|Z_0|=a_1+2a_2+3a_3+4c_4$ and $|E(G[Z_0])|=a_2+2a_3+4c_4$, we obtain
\begin{align}
 3a_1+4a_2+5a_3+4c_4 =3|Z_0|-2|E(G[Z_0])| \leq e_{\circ}+e_{\bullet}+2\gamma +3b_{\geq2}+\delta(T).
 \label{eq:a}
\end{align}  

Now, recalling the respective definitions and applying the degree bound $d_G(x)\leq3$ gives, $e_{\circ}\leq3a_1$, $e_{\bullet}\leq4a_2$, and $\gamma\leq a_3$. Also, defining $b_1:=|\{v\in V(G):\mathrm{wt}(v)=1\}|$, Lemma~\ref{lem:unit} gives, $e_{\circ}+2e_{\bullet}+2\gamma\leq2b_1$. This is because each weight-one vertex contributes at most two
to the LHS, by the three cases of the lemma. Combining the above, we obtain
\begin{align*}
 7(e_{\circ}+e_{\bullet}+2\gamma) =3(e_{\circ}+2e_{\bullet}+2\gamma) +4e_{\circ}+e_{\bullet}+8\gamma \leq 6b_1+12a_1+4a_2+8a_3.
\end{align*}
Multiplying \eqref{eq:a} by seven and applying the above bound gives, 
$$9a_1+24a_2+27a_3+28c_4 \leq6b_1+21b_{\geq2}+7\delta(T)  .  $$
Using $|Z_0|=a_1+2a_2+3a_3+4c_4$ and rearranging, we obtain 
$$9|Z_0| \leq6b_1+21b_{\geq2}+7\delta(T)+8c_4-6a_2.$$ 
Since  $\overline W = W-b_1-2b_{\geq2}$ and $\lambda(T)=|Z_0|+b_1+b_{\geq2}$, adding $9b_1+9b_{\geq2}$ to both sides of the preceding
inequality gives
\begin{equation}\label{eq:incidence}
9\lambda(T) \leq15W+7\delta(T)+8c_4-15\overline W-6a_2.
\end{equation}
Now, substituting \eqref{eq:edge} and \eqref{eq:W} into \eqref{eq:incidence} and rearranging gives \eqref{eq:points}. 
\end{proof}

\subsubsection{Bounding the Number of Zero $4$-cycles}

By Observation~\ref{obs:dual}(2), each zero $4$-cycle in the dual $G$ corresponds to four triangular faces surrounding an interior vertex of $S_1$ of degree four in the primal. We will refer to this configuration as a \emph{zero wheel} and the four boundary edges of the wheel as its \emph{rim edges}. Recall  that there are $c_4$ zero wheels, with distinct centers.  Each rim edge either lies on the convex hull or is incident with a triangular face outside the wheel. By Observation~\ref{obs:z} (2), this outside face has weight at least two. Further, distinct zero wheels cannot share a rim edge, since their incident zero faces would then belong to the same zero component. The following observation shows that a rim edge of a zero wheel is nonflippable whenever its incident face outside the wheel has weight two. Recall that an interior edge $(u, v)$ of a straight-line triangulation of a planar point set is
\emph{flippable} if its two incident triangular faces together form a convex quadrilateral. On the other hand, an interior edge whose two incident faces form a nonconvex quadrilateral is said to be \emph{nonflippable}.

\begin{figure}[tbp]
  \centering
  \begin{minipage}[t]{.49\textwidth}
    \centering
    \begin{tikzpicture}[scale=1.04]
      \coordinate (u) at (-1.15,0);
      \coordinate (s) at (-1.05,1.9);
      \coordinate (r) at (1.05,2.15);
      \coordinate (v) at (1.15,.12);
      \coordinate (p) at (.05,.95);
      \coordinate (x) at (.03,-1.12);
      \fill[propTeal!13] (x)--(u)--(s)--(r)--(v)--(p)--cycle;
      \draw[black!40,line width=.55pt]
        (u)--(s)--(r)--(v)--(x)--(u)
        (p)--(u) (p)--(s) (p)--(r) (p)--(v);
      \draw[black!45,densely dashed] (u)--(v);
      \draw[propTeal,line width=1.25pt]
        (x)--(u)--(s)--(r)--(v)--(p)--cycle;
      \foreach \q in {u,s,r,v,p,x} \node[propvertex] at (\q) {};
      \node[left] at (u) {$p_1$}; \node[left] at (s) {$p_4$};
      \node[right] at (r) {$p_3$}; \node[right] at (v) {$p_2$};
      \node[right=2pt] at (p) {$p$}; \node[below] at (x) {$q$};
      \fill[propBlue] (-.30,-.42) circle[radius=2pt];
      \fill[propBlue] (.33,-.42) circle[radius=2pt];
      \node[font=\small] at (0,-1.77)
        {$Q_1=(q, p_1, p_4, p_3, p_2, p)$};
    \end{tikzpicture}
  \end{minipage}\hfill
  \begin{minipage}[t]{.49\textwidth}
    \centering
    \begin{tikzpicture}[scale=1.04]
      \coordinate (u) at (-1.15,0);
      \coordinate (s) at (-1.05,1.9);
      \coordinate (r) at (1.05,2.15);
      \coordinate (v) at (1.15,.12);
      \coordinate (p) at (.05,.95);
      \coordinate (x) at (.03,-1.12);
      \fill[propPurple!13] (x)--(v)--(r)--(s)--(u)--(p)--cycle;
      \draw[black!40,line width=.55pt]
        (u)--(s)--(r)--(v)--(x)--(u)
        (p)--(u) (p)--(s) (p)--(r) (p)--(v);
      \draw[black!45,densely dashed] (u)--(v);
      \draw[propPurple,line width=1.25pt]
        (x)--(v)--(r)--(s)--(u)--(p)--cycle;
      \foreach \q in {u,s,r,v,p,x} \node[propvertex] at (\q) {};
      \node[left] at (u) {$p_1$}; \node[left] at (s) {$p_4$};
      \node[right] at (r) {$p_3$}; \node[right] at (v) {$p_2$};
      \node[right=2pt] at (p) {$p$}; \node[below] at (x) {$q$};
      \fill[propBlue] (-.30,-.42) circle[radius=2pt];
      \fill[propBlue] (.33,-.42) circle[radius=2pt];
      \node[font=\small] at (0,-1.77)
        {$Q_2=(q, p_2, p_3, p_4, p_1, p)$};
    \end{tikzpicture}
  \end{minipage}
  \caption{\small{ Illustration for the proof of Observation \ref{obs:boundary}. }}
  \label{fig:flippable-wheel}
\end{figure}
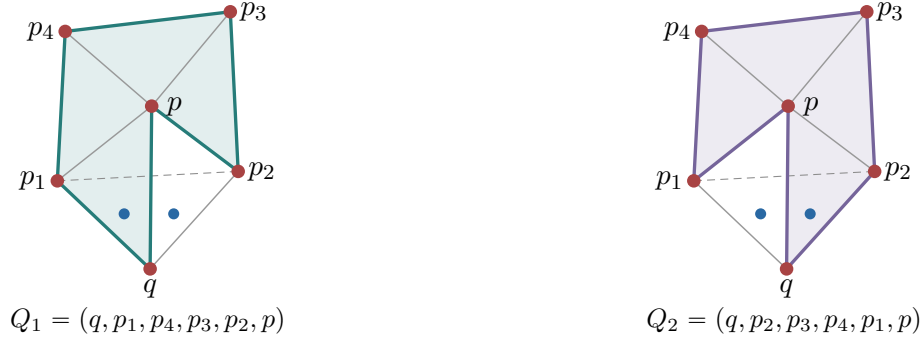

\begin{observation}\label{obs:boundary}
A rim edge of a zero wheel is nonflippable in the triangulation $T$, whenever its incident face outside the wheel has weight two.
\end{observation}

\begin{proof}
Let $U$ be a zero wheel with central vertex $p$ and rim vertices $(p_1, p_2, p_3, p_4)$, labeled cyclically. Suppose, for a contradiction, that the rim edge $p_1p_2$
is flippable. Let $(p_1,p_2, q)$ be its outside face. By Observation \ref{obs:z} (2), $q$ is distinct from the vertices of the wheel: otherwise, the outside face would share
rim edges with two wheel faces. The other three wheel faces form an empty pentagon
with boundary $(p_1,p_4,p_3,p_2,p)$. Flipping $p_1p_2$ produces the triangles $(p, q, p_1)$ and $(p, q, p_2)$. Attaching either triangle to this pentagon gives one of the two monochromatic hexagons $Q_1=(q,p_1,p_4,p_3,p_2,p)$ and $Q_2=(q,p_2,p_3,p_4,p_1,p)$, as shown in Figure~\ref{fig:flippable-wheel}. Since the outside face $(p_1, p_2, q)$ has at most two interior points, at least one of $Q_1$ and $Q_2$ contains at most one interior point, which is a contradiction.  
\end{proof}

\begin{figure}[htbp]
\centering
\begin{tikzpicture}[
  font=\small,
  vertex/.style={circle,fill=red!65!black,inner sep=1.6pt},
  bluepoint/.style={circle,fill=blue!65!black,inner sep=1.7pt},
  rim/.style={violet!80!black,dashed,line width=1.2pt},
  anglemark/.style={black!60,line width=.55pt}
]
\begin{scope}[scale=.9]
  \coordinate (v) at (0,0);
  \coordinate (u) at (3,0);
  \coordinate (p) at (115:2.3);
  \coordinate (x) at (-120:2.1);
  \coordinate (s) at (.5,4.5);
  \coordinate (r) at (-3.2,2);

  \fill[teal!6] (u)--(s)--(r)--(v)--cycle;
  \fill[teal!18] (p)--(v)--(u)--cycle;
  \fill[orange!18] (x)--(u)--(v)--cycle;
  \draw[black!65,line width=.8pt] (u)--(s)--(r)--(v);
  \draw[black!55,line width=.65pt] (p)--(s) (p)--(r);
  \draw[black!80,line width=1.1pt]
    (p)--(u)--(x)--(v)--cycle;
  \draw[rim] (u)--(v);
  \node[above,font=\footnotesize,text=violet!80!black]
    at (1.45,0) { rim edge };
%
%
  \node at (1.45,-.32) {$F$};
  \node[bluepoint] at (.25,-.90) {};
  \node[bluepoint] at (.85,-.70) {};
%
%
  \node[vertex,label=left:$p_2$] at (v) {};
  \node[vertex,label=right:$p_1$] at (u) {};
  \node[vertex,label=below right:$p$] at (p) {};
  \node[vertex,label=above:$p_3$] at (s) {};
  \node[vertex,label=left:$p_4$] at (r) {};
  \node[vertex,label=below left:$q$] at (x) {};
\end{scope}

\begin{scope}[shift={(6.9,1.7)},scale=1.55]
  \coordinate (v) at (0,0);
  \coordinate (a) at (35:2.3);
  \coordinate (b) at (145:2.3);
  \coordinate (p1) at (-75:2.1);
  \coordinate (p2) at (255:2.0);

  \fill[orange!18] (v)--(a)--(b)--cycle;
  \fill[teal!13] (v)--(p1)--(a)--cycle;
  \fill[teal!13] (v)--(b)--(p2)--cycle;
  \draw[black!80,line width=.9pt]
    (a)--(b)--(p2)--(v)--(p1)--cycle;
  \draw[rim] (a)--(v) (b)--(v);

  \node at (0,1.03) {$F$};
  \node at (.73,-.84) {$\Delta_1$};
  \node at (-.71,-.80) {$\Delta_2$};
  \node[bluepoint] at (-.65,.97) {};
  \node[bluepoint] at (.65,.97) {};
%
%
  \node[vertex,label=below:$r$] at (v) {};
  \node[vertex,label=above right:$a$] at (a) {};
  \node[vertex,label=above left:$b$] at (b) {};
  \node[vertex,label=below right:$p_1$] at (p1) {};
  \node[vertex,label=below left:$p_2$] at (p2) {};
\end{scope}

\node[align=center,text width=5.8cm,anchor=north] at (0,-2.7)
  { (a) };
\node[align=center,text width=5.8cm,anchor=north] at (6.9,-2.7)
  { (b) };
\end{tikzpicture}
\caption{ \small{  Illustration for the proof of Lemma \ref{lem:cycle}: (a) A zero wheel $(p_1, p_2, p_3, p_4)$ with central vertex $p$ and rim edge $(p_1, p_2)$ adjacent to an outside face $F$ with two points of color 2 in the interior and (b) two nonconvex  quadrilaterals with common reflex vertex $r$ share their weight-two face $F$. } }
\label{fig:cycle}
\end{figure}
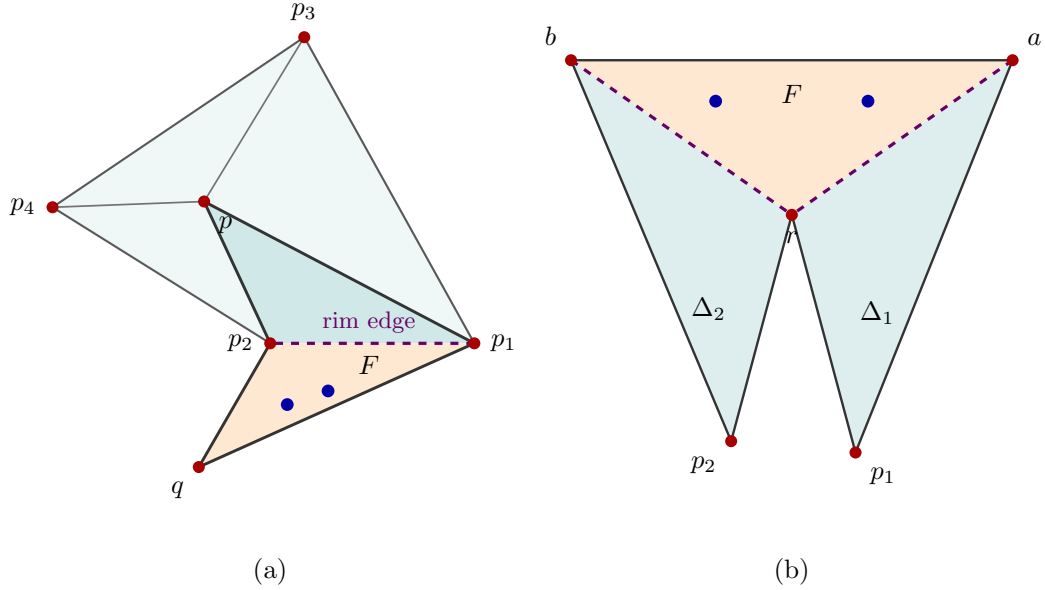

The next lemma establishes the desired bound on $c_4$, which is the number of  zero wheels with distinct centers.   

\begin{lemma}\label{lem:cycle} 
$6c_4\leq2|S_1|-\delta(T)+3\overline W$, where 
$\overline W$ is as defined in \eqref{eq:points}.   
\end{lemma}

\begin{proof}
Let $J_2$ and $J_{\geq3}$ count the non-hull rim edges of zero wheels whose outside faces have weight two and at least three, respectively.   

For each rim edge counted by $J_2$, consider the quadrilateral
formed by its two incident faces. By Observation~\ref{obs:boundary}, this quadrilateral is nonconvex. Further, its unique reflex vertex is an endpoint of the rim edge (see, for example, the quadrilateral $(p, p_1, q, p_2)$ in Figure \ref{fig:cycle} (a)). Note that the reflex vertex can neither lie on the convex hull nor can it be a zero-wheel center, which is incident only with zero faces.  Denote by $R$ the set of interior vertices of $T$ that are not centers of a zero wheel. The following facts about $R$ are easy to verify: 

\begin{itemize}

\item $|R| =  |S_1|-\delta(T)-c_4$. This is because $R$ excludes from $S_1$, \(\delta(T)\) hull vertices and the \(c_4\) zero-wheel centers.

\item Each vertex $r \in R$ can be the reflex vertex of at most two nonconvex quadrilaterals as described above. This is because each nonconvex quadrilateral consists of a zero face and a weight-two face whose angles at $r$ sum to more than $\pi$. Since each face of a zero wheel has exactly one rim edge, distinct quadrilaterals have distinct zero faces. 
Hence, any pair of nonconvex quadrilaterals with reflex vertex $r$ must share their weight-two face, otherwise their four distinct face angles would sum to more than $2\pi$. Therefore, all such quadrilaterals share a single weight-two face $F$, as in Figure \ref{fig:cycle} (b). Their rim edges are sides of $F$ incident with $r$. Since a triangle has only two such sides, there are at most two quadrilaterals with reflex vertex $r$.  

\end{itemize}
For each $r\in R$, let $m(r)$ count the rim edges whose associated quadrilateral has reflex vertex $r$. Every edge counted by $J_2$ contributes to exactly one $m(r)$, and the preceding argument shows that $m(r)\leq2$. Hence, $J_2=\sum_{r\in R}m(r) \leq2|R| = 2 (|S_1|-\delta(T)-c_4 )$.

Also,  every face of weight at least three contributes at least one to $\overline W$ and can be the outside face of at most three rim edges. Hence, $J_{\geq3}\leq3\overline W$. The zero wheels have $4c_4$ distinct rim edges, at most
$\delta(T)$ of which lie on the convex hull. Combining the above gives
$$ 4c_4  \leq\delta(T)+J_2+J_{\geq3} \leq\delta(T) +2\left(|S_1|-\delta(T)-c_4\right)+3\overline W.$$
Rearranging proves the result in Lemma \ref{lem:cycle}.  
\end{proof}

\subsubsection{Completing the Proof of
Proposition~\ref{prop:shull}}

Multiplying \eqref{eq:points} by three and applying Lemma \ref{lem:cycle}, we obtain  
\begin{align*}
 9|S_1|
 &\leq48\delta(T)+54+24c_4-45\overline W-18a_2\\
 &\leq48\delta(T)+54
      +4\left(2|S_1|-\delta(T)+3\overline W\right)
      -45\overline W-18a_2\\
 &=8|S_1|+44\delta(T)+54-33\overline W-18a_2.
\end{align*}
Since $\overline W,a_2\geq0$, it follows that
\[
 |S_1|
 \leq44\delta(T)+54-33\overline W-18a_2
 \leq44\delta(T)+54,
\]
contradicting \eqref{eq:shull}. This completes the proof
of Proposition~\ref{prop:shull}. \hfill$\Box$

\section{Proof of Theorem \ref{thm:r}  }\label{sec:rpf}

\subsection{Proof of the Upper Bound}

Fix $r\geq3$, $c\geq3$, and let $\delta_c$ be as defined in the statement of Theorem \ref{thm:r}. Define 
$$s_c:=(r-2)(c-2)+\left\lfloor(r-2)\delta_c\right\rfloor.$$
Note that
$$s_c+1-(r-2)(c-2+\delta_c) = 1+\left\lfloor(r-2)\delta_c\right\rfloor
-(r-2)\delta_c >0.$$
Choose an integer $h_1\geq r$ satisfying
\begin{equation}\label{eq:convexh}
h_1> 2(r-2)+ \frac{(r-2)(c-1)\left( r-1+(2r-7)\delta_c \right ) }{ s_c+1-(r-2)(c-2+\delta_c)}.
\end{equation}
We show that $\mathsf{M}_r^\star(c,s_c)\leq c \mathsf{ES}(h_1)$. Suppose, for a contradiction, that $S$ is a set of at least $c\,\mathsf{ES}(h_1)$ points in general position, colored with $c$ colors, such that every monochromatic simple $r$-gon
contains at least $s_c+1$ points of $S$ in its interior.
Let $S_1,\ldots,S_c$ be the color classes. By the pigeonhole principle, some color class contains at least $\mathsf{ES}(h_1)$ points. Hence, the Erd\H{o}s--Szekeres
theorem guarantees a monochromatic convex $h_1$-gon. Among all monochromatic convex $h_1$-gons determined by $S$, choose one with the minimum number of points of $S$ in its interior. Denote this by $T$. Relabeling the colors if
necessary, assume that its vertices have color $1$. Now, for $1\leq a \leq c$, define
$$T_a:=CH(T)\cap S_a.$$ Let $i_1$ be the number of color $1$ points inside $T$. Thus, 
$$|T_1|=h_1+i_1 \quad \text{ and } \quad |V(CH(T_1))|=h_1.$$
Also, by the minimality of $T$, 
\begin{equation}\label{eq:general-inner-hulls}
|V(CH(T_a))|\leq h_1-1  ,  
\end{equation}
for $2\leq a \leq c$, whenever $|T_a|\geq3$. We now apply Lemma~\ref{lem:r} separately to the color classes as follows: 

\begin{itemize}

\item First, we apply Lemma~\ref{lem:r} to $T_1$. This gives at least
$$\left\lfloor\frac{h_1+i_1-2}{r-2}\right\rfloor + \left\lfloor\frac{i_1}{r-2}\right\rfloor \geq
\frac{h_1+2i_1-2r+4}{r-2}$$
pairwise interior-disjoint $r$-gons whose vertices have color $1$ and whose interiors contain no points of color 1. Each of these $r$-gons contains at least $s_c+1$
points of other colors. Since their interiors are disjoint,
$$\sum_{a=2}^c |T_a| \geq \frac{(s_c+1)(h_1+2i_1-2r+4)}{r-2}.$$
Rearranging the terms gives, 
\begin{equation}\label{eq:outercolor}
-(r-2)\sum_{a=2}^c |T_a|+2(s_c+1)i_1 \leq -(s_c+1)h_1+2(s_c+1)(r-2).
\end{equation}

\item Next, we fix $a \in\{2,\ldots,c\}$ and apply Lemma~\ref{lem:r} to $T_a$ with only the first summand. Specifically, if $|T_a|\geq r$, then applying Lemma~\ref{lem:r} to $T_a$ by dropping the second term gives at least $$\left\lfloor\frac{|T_a|-2}{r-2}\right\rfloor \geq \frac{|T_a|-r+1}{r-2}$$
pairwise interior-disjoint $r$-gons of color $a$ whose interiors contain no points of color $a$. Each of these $r$-gons contains at least $s_c+1$ points of other colors. Hence, there are at least 
$$(s_c+1)\frac{|T_a|-r+1}{r-2}$$
points inside $CH(T_a)$. Since $a\geq2$, the convex hull $CH(T_a)$ lies strictly
inside $T$. The total number of points of colors different from $a$ available in their interiors
is therefore at most $i_1+\sum_{b=2}^c |T_b|-|T_a|.$ Hence, it follows that 
\begin{equation}\label{eq:P}
(r-2)\left(i_1+\sum_{b=2}^c |T_b|-|T_a|\right)
\geq
(s_c+1)(|T_a|-r+1).
\end{equation}
If $|T_a|<r$, this inequality holds trivially because its RHS is nonpositive. Adding \eqref{eq:P} over $a \in \{2,\ldots,c\}$ and rearranging the terms gives,   
\begin{equation}\label{eq:innercolorP}
\left[s_c+1-(c-2)(r-2)\right]\sum_{b=2}^c |T_b| -(c-1)(r-2)i_1 \leq (s_c+1)(c-1)(r-1).  
\end{equation}

\item Finally, we apply the full bound in Lemma~\ref{lem:r} to $T_a$. Specifically, for $|T_a|\geq r$, applying \eqref{eq:general-inner-hulls} gives,  
\begin{align*}
\left\lfloor\frac{|T_a|-2}{r-2}\right\rfloor + \left\lfloor \frac{|T_a|-|V(CH(T_a))|}{r-2} \right \rfloor  \geq \frac{2|T_a|-|V(CH(T_a))|-2r+4}{r-2} \geq \frac{2|T_a|-h_1-2r+5}{r-2}.
\end{align*} 
Counting the points in these polygons as before yields
\begin{equation}\label{eq:Tboth}
(r-2)\left(i_1+\sum_{b=2}^c |T_b|-|T_a|\right)
\geq
(s_c+1)(2|T_a|-h_1-2r+5).
\end{equation}
This inequality also holds when $|T_a|<r$, since then $2|T_a|-h_1-2r+5\leq3-h_1\leq0$.
Adding \eqref{eq:Tboth} over $a \in \{ 2, 3, \ldots, c\}$ gives, 
\begin{equation}\label{eq:innercolorboth}
\left[2(s_c+1)-(c-2)(r-2)\right]\sum_{b=2}^c |T_b| -(c-1)(r-2)i_1 \leq
(s_c+1)(c-1)(h_1+2r-5).
\end{equation}
\end{itemize}

We now combine the inequalities \eqref{eq:outercolor}, \eqref{eq:innercolorP}, and \eqref{eq:innercolorboth} as follows:  Multiply \eqref{eq:outercolor} by $(c-1)\delta_c$, multiply \eqref{eq:innercolorP} by $1-2\delta_c$,
and multiply \eqref{eq:innercolorboth} by $\delta_c$. Note that $\delta_c \in (0, \frac{1}{2}]$, hence all multipliers are nonnegative. This leads to the following simplifications: 
\begin{itemize}

\item The terms involving $h_1$ cancel. 

\item The coefficient of $\sum_{b=2}^c |T_b|$ simplifies to $s_c+1-(r-2)(c-2+\delta_c)$.

\item Using the identity $2\delta_c^2+(2c-3)\delta_c=1$, the coefficient
of $i_1$ simplifies to, 
\begin{align*}
(c-1)\left[2(s_c+1)\delta_c-(r-2)(1-\delta_c)\right] = 2(c-1)\delta_c
\left[s_c+1-(r-2)(c-2+\delta_c)\right].
\end{align*}

\item The constant terms on the RHS of the inequalities simplify to $(s_c+1)(c-1)\left[r-1+(2r-7)\delta_c\right]$.
\end{itemize} 
Combining the above gives, 
$$\left[s_c+1-(r-2)(c-2+\delta_c)\right] \left[ \sum_{b=2}^c |T_b|+2(c-1)\delta_c i_1
\right]  \leq  (s_c+1)(c-1)\left[r-1+(2r-7)\delta_c\right].$$
Since $i_1\geq0$ and
$s_c+1-(r-2)(c-2+\delta_c)>0$, it follows that
\[
\sum_{b=2}^c |T_b|
\leq
\frac{
(s_c+1)(c-1)\left[r-1+(2r-7)\delta_c\right]
}{
s_c+1-(r-2)(c-2+\delta_c)
}.
\]
On the other hand, \eqref{eq:outercolor} and $i_1\geq0$ imply, 
$$h_1 \leq \frac{r-2}{s_c+1} \sum_{b=2}^c |T_b| + 2(r-2).$$
Combining the above inequalities gives, 
$$h_1\leq 2(r-2)+ \frac{(r-2)(c-1)\left[r-1+(2r-7)\delta_c\right]}{ s_c+1-(r-2)(c-2+\delta_c)},$$
contradicting \eqref{eq:convexh}.

\subsection{Proof of the Lower Bound} 
\label{sec:lowerboundpf}

Let $A$ and $B$ be two finite point sets in the plane with pairwise distinct $x$-coordinates. We say that $A$ lies \emph{high above} $B$ if every point of $A$ lies above every line determined by two points of $B$, and every point of $B$ lies below every line determined by two points of $A$.
Using this notion, Horton sets are defined recursively \cite{horton1983sets} as follows (see also \cite[Chapter 3]{discretegeometry}).

\begin{definition}\label{defn:H}
A finite point set with distinct
$x$-coordinates $H=\{p_1,p_2,\ldots,p_n\}$, labeled in increasing order of their $x$-coordinates, is called a {\it Horton set} if the following conditions hold recursively:
\begin{itemize}
    \item if $|H|\leq 2$, then $H$ is a Horton set, 
    \item if $|H|\geq 3$, then the sets
    $$H^{-} :=  \{p_1,p_3,p_5,\ldots\} \quad\text{ and } \quad H^{+} := \{p_2,p_4,p_6,\ldots\}$$ are both Horton sets, and $H^{+}$ lies high above $H^{-}$. The sets $H^{+}$ and $H^{-}$ will be referred to as the {\it upper child} and {\it lower child} of $H$, respectively.  
\end{itemize} 
\end{definition}

First suppose that $c=2q+1$ is odd. Let $H=\{h_1,\ldots,h_n\}$ be a Horton set, indexed in increasing order of $x$-coordinates. Let $\phi$ be a coloring of the points of $H$ defined as follows (see Figure \ref{fig:pointsh}):   
\[
 \phi(h_s)=1+((s-1)\bmod c).
\]
Hence, the colors repeat cyclically in the order $1,2,\ldots,c$. If $B=\{b_1,\ldots,b_m\}$ is a recursive subset of $H$ at depth $d$, listed in increasing $x$-order, then $ b_j=h_{a+2^d(j-1)}$ for some integer $a \geq 1$. Since $c$ is odd, 
\begin{equation}\label{eq:horton-color-period}
 \phi(b_i)=\phi(b_j)
 \quad\Longleftrightarrow\quad
 i\equiv j\pmod c.
\end{equation}
In particular, the coloring on every recursive subset is periodic with period $c$, up to a permutation of the colors. We also use the following geometric property of Horton sets: if $u$ and $v$ belong to different children of a recursive subset $D$, then every point of $D^+$ whose $x$-coordinate lies strictly between the $x$-coordinates of $u$ and $v$ lies above the line $uv$, while every such point of $D^-$ lies below that line (see, for example, the proof of Lemma 4.2 in \cite{interiorpoints}).

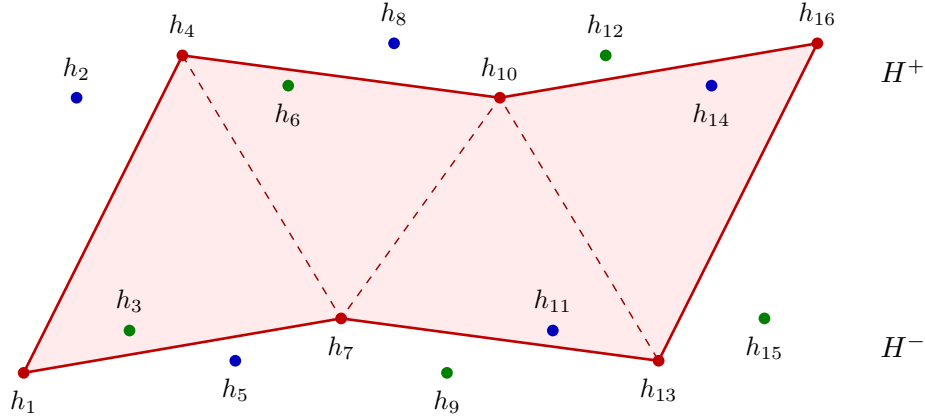
\begin{figure}[h]
\centering
\begin{tikzpicture}[
  x=.70cm,y=.040cm,font=\small,
  color1/.style={red!75!black},
  color2/.style={blue!75!black},
  color3/.style={green!50!black},
  hexagon/.style={draw=red!75!black,line width=1pt,
                   fill=red!8},
  diagonal/.style={draw=red!60!black,dashed,line width=.55pt}
]
\foreach \i/\y in {
  1/0,2/91,3/14,4/105,5/4,6/95,7/18,8/109,
  9/0,10/91,11/14,12/105,13/4,14/95,15/18,16/109}
  \coordinate (h\i) at ({\i-1},\y);

\path[hexagon] (h1)--(h4)--(h10)--(h16)--(h13)--(h7)--cycle;
\draw[diagonal] (h4)--(h7)--(h10)--(h13);
%
%
\foreach \i in {1,...,16}{
  \pgfmathtruncatemacro{\pointcolor}{mod(\i-1,3)+1}
  \fill[color\pointcolor] (h\i) circle[radius=2.2pt];
}

\foreach \i/\pos in {
  1/below,2/above,3/above,4/above,5/below,6/below,
  7/below,8/above,9/below,10/above,11/above,12/above,
  13/below,14/below,15/below,16/above}
  \node[\pos=3pt] at (h\i) {$h_{\i}$};

\node[anchor=west,font=\normalsize] at (16,100) {$H^{+}$};
\node[anchor=west,font=\normalsize] at (16,9) {$H^{-}$};
%
%
\end{tikzpicture}
\caption{ \small{  A Horton set of size 16 colored with 3 colors in a cyclic manner as above. The monochromatic hexagon $(h_1,h_4,h_{10},h_{16},h_{13},h_7)$ contains exactly the four points $h_3,h_6,h_{11},h_{14}$ in the interior.  }  }  
\label{fig:pointsh}
\end{figure}

\begin{observation}  
\label{obs:q} 
Every monochromatic triangle in a Horton set $H$ colored with $c= 2q+1$ colors as above, contains at least $q$ points of $H$ in its interior. 
\end{observation}

\begin{proof}  
To prove this, let $\triangle uvw$ be monochromatic, and choose the deepest recursive subset $B$ containing the three vertices. By the choice of $B$, one child $A$ contains
two of the vertices, suppose they are $u, v$, and the other child $A'$ contains $w$. Choose the deepest recursive subset $D\subseteq A$ containing both $u$ and $v$, and suppose that their positions in the $x$-ordering of $D$ differ by $t$. By \eqref{eq:horton-color-period}, $c$ divides $t$, and hence, $t\geq c$. Also, by the choice of $D$, the vertices $u$ and $v$ belong to different children, which means $t$ is odd.  There are exactly $t-1$ points of $D$ strictly between $u$ and $v$ in $x$-order, with $(t-1)/2$ in each child. The geometric property stated above places these two groups on opposite sides of the line $uv$. Hence, at least $\frac{t-1}{2}\geq\frac{c-1}{2}=q$ points of $D$ lie between $u$ and $v$ in $x$-order and on the same side of $uv$ as $w$. By the property of a Horton set, for any such point $z$ that lies above $uv$, the point $w$ lies above both $uz$ and $zv$. These conditions imply that $z$ lies in the interior of $\triangle uvw$. The case in which $w$ lies below the line $uv$ follows by the same argument, with above and below interchanged. This proves the result in Observation \ref{obs:q}.  
\end{proof}

Now let $P$ be any monochromatic simple $r$-gon determined by $H$. Triangulate $P$ using diagonals joining its vertices (see Figure \ref{fig:pointsh}). This produces $r-2$ monochromatic triangles
$\Delta_1,\ldots,\Delta_{r-2}$ with pairwise disjoint interiors, all contained in $P$. By Observation \ref{obs:q},  
$$|H\cap\operatorname{int}(P)| = \sum_{i=1}^{r-2} |H\cap\operatorname{int}(\Delta_i)| \geq (r-2)q.$$
Finally, suppose that $c=2q+2$ is even. Take a Horton set $H=\{h_1,\ldots,h_n\}$ and let $H'=\{h_1,\ldots,h_{n-1}\}$, which is also a Horton set.  Color $H'$ cyclically with $2q+1$ colors as above, and assign the point  $h_n$ the remaining color. Since the last color occurs only once, every monochromatic $r$-gon, with $r\geq3$, has all its vertices in $H'$. By the odd-color case, it contains at least $(r-2)q$ points of $H'$, and hence at least that many points of $H$. In both cases, $q=\lfloor\frac{c-1}{2}\rfloor$. Since Horton sets exist with arbitrarily many points, this implies,  
$$\mathsf{M}_r^\star\left(  c,  (r-2)\left\lfloor\frac{c-1}{2}\right\rfloor-1 \right)=\infty.$$
This proves the lower bound on $\lambda_r^\star(c)$ in Theorem \ref{thm:r}.  \hfill  $\Box$

\section{Conclusions} 
\label{sec:conclusions}

In this paper, we study monochromatic polygons with few interior points in colored planar point sets. Improving the bounds on $\lambda_r^\star(c)$ remains an interesting direction for further research. The bichromatic case is particularly intriguing. As already mentioned, it is possible that $\lambda_r^{\star}(2)=0$ for every fixed $r\geq3$; that is, every sufficiently large $2$-colored planar point set
in general position contains a monochromatic empty $r$-gon, not necessarily convex. Proposition~\ref{prop:H} establishes this assertion for sufficiently large Horton sets. The first unresolved case is $r=5$, which leads to the following question. Since every simple pentagon is $2$-convex, this question is equivalently phrased in terms of the existence of empty monochromatic $2$-convex pentagons in
sufficiently large bichromatic point sets (see \cite[Problem 4.1]{aichholzer2009color}).

\begin{question}\label{question:fivetwo}
Does every sufficiently large $2$-colored planar point set in general
position contain a monochromatic empty $5$-gon, not necessarily convex?
\end{question}

Finally, for each fixed $r$, our lower and upper bounds on $\lambda_r^\star(c)$ have leading coefficients of $c$ equal to $\frac{r-2}{2}$ and $r-2$, respectively, leaving a gap of a factor of two. This gap persists even for triangles (recall \eqref{eq:almostempty3}), where the correct leading coefficient remains unknown. For large $c$, iterating the nested convex-hull argument over additional steps may improve the additive correction in the upper  bound, as in the case of triangles~\cite{interiorpoints}. However, this approach does not appear to improve the leading coefficient of $c$. Determining the asymptotic behavior of $\lambda_r^\star(c)$ and, in particular, closing the gap between the leading constants remains an interesting open problem.

\small

\subsection*{Acknowledgements} BBB was supported by NSF CAREER grant DMS 2046393, a Sloan Research Fellowship, and the National University of Singapore's PESS and Provost's Chair funds.

\subsection*{AI disclosure} GPT 5.6 Sol and GPT 6 Astra were used to brainstorm proof ideas, verify some of the mathematical arguments, refine grammar and typography, and prepare the figures.  The authors assume full responsibility for all content.

\bibliographystyle{abbrvnat} 
\bibliography{refIII.bib}

\normalsize

\appendix

\section{ Proof of Proposition \ref{prop:H}}
\label{sec:Hpf}

Throughout, $H=\{p_1, \ldots, p_m\}$ will be a 2-colored Horton set with points in increasing order of their $x$-coordinates. 

\begin{definition}
A pair of points $p_u, p_v \in H$, with $1 \leq u < v \leq m$, is said to be 
\emph{upper-visible} in $H$ if every point $p_w \in H$, with $1 \leq u< w < v \leq m$,
lies strictly below the line joining $p_u$ and $p_v$.
\end{definition}

Using the notion upper-visibility, we construct monochromatic chains in $H$ as follows:

\begin{lemma}
\label{lem:chain}
Let $L\geq2$ be an integer. Every $2$-colored Horton set $H$ with at least $2^{2L-1}$ points contains $L$ monochromatic points in increasing order of their $x$-coordinates whose consecutive edges are upper-visible in $H$.
\end{lemma}

\begin{proof}
Define $H_0:=H$ and $H_t:=(H_{t-1})^+$, for $1\leq t\leq2L-2$.
Thus, each set is the upper child of the preceding set. Observe that every edge joining consecutive points of $H_t$ is upper-visible in $H$. Since $|H|\geq2^{2L-1}$, 
$$|H_{2L-2}| \geq \left\lfloor\frac{|H|}{2^{2L-2}}\right\rfloor \geq 2.$$
Choose consecutive points $p_u,p_v$ in this set, with $u < v$.
Initialize the singleton left and right chains
$\mathcal C^{-}_{2L-2}:=(p_u)$ and
$\mathcal C^{+}_{2L-2}:=(p_v)$, respectively.
For $t\in\{2L-2,2L-3,\ldots,1\}$, we maintain two increasing
monochromatic chains in $H$:
$$\mathcal C^{-}_t=(p_{x_1},\ldots,p_{x_a}) \quad\text{ and }\quad \mathcal C^{+}_t=(p_{y_1},\ldots,p_{y_{b}}), $$
with current endpoints $p_{x_a}=p_u$ and $p_{y_1}=p_v$, such that every edge joining consecutive vertices of either chain is upper-visible in $H$, and $p_u, p_v$ are consecutive in $H_t$.

Having constructed the chains $\mathcal C^{-}_t$ and $\mathcal C^{+}_t$, let $p_k$ be the unique point of $H_{t-1}$ between $p_u$ and $p_v$ in order of their $x$-coordinates. Both edges $(p_u,p_k)$ and $(p_k,p_v)$ join consecutive points of $H_{t-1}$, hence, both of these pairs are upper-visible in $H$. Apply the first applicable case below: 

\begin{itemize}

\item If $p_k$ has the color of $p_v$, update the chains as follows:
$$
\mathcal C^{-}_{t-1}:=\mathcal C^{-}_t \quad\text{ and }\quad \mathcal C^{+}_{t-1}=(p_k,p_{y_1},\ldots,p_{y_{b}}).$$
The new endpoints are $(p_u,p_k)$.

\item Otherwise, if $p_k$ has the color of $p_u$, update the
chains as follows:
$$
\mathcal C^{-}_{t-1}=(p_{x_1},\ldots,p_{x_a},p_k) \quad\text{ and }\quad
\mathcal C^{+}_{t-1}=\mathcal C^{+}_t. $$
The new endpoints are $(p_k,p_v)$.

\item In the remaining case, $p_u$ and $p_v$ have the same color, different from that of $p_k$. In this case, $\mathcal C^{-}_t=(p_u)$. This is because, once the endpoints have different colors, the first two cases preserve that property. Thus, before the present case, the endpoints have always had the same color, and every update has used the first case, leaving the left chain unchanged. Then update as follows:
$$\mathcal C^{-}_{t-1}:=(p_k) \quad\text{ and }\quad \mathcal C^{+}_{t-1}:=\mathcal C^{+}_t.$$
The new endpoints are $(p_k, p_v)$.

\end{itemize}

Note that each of the above updates preserves the chain properties and makes the new endpoints consecutive in $H_{t-1}$, hence, the construction continues. Also, note that the third case occurs at most once, since it makes the colors of the endpoints different, and it preserves the total number of chain vertices. In both the other cases, the number of chain vertices increase by one. Hence, after $2L-2$ iterations,  
$$|\mathcal C^{-}_0|+|\mathcal C^{+}_0| \geq 2+(2 L - 2)-1=2L-1.$$
Then, one of these chain has at least $L$ vertices. Considering $L$ consecutive vertices of that chain proves the result in Lemma \ref{lem:chain}.    
\end{proof}

The next observation shows that an upper visible chain from $H^-$ can be combined with a point in $H^+$ to obtain an empty $r$-gon.

\begin{observation}\label{obs:rtriangle}
Suppose $p_{a_1},\ldots, p_{a_{r-1}}\in H^{-}$, with $a_1 < a_2 < \ldots, a_{r-1}$, and every pair of points $p_{a_i}, p_{a_{i+1}}$ is upper-visible in $H^{-}$, for $1 \leq i \leq r - 2$. Then for any $p_b\in H^+$, $(p_b, p_{a_1}, p_{a_2},\ldots, p_{a_{r-1}})$ is an empty $r$-gon in $H$.
\end{observation}

\begin{proof}
Note that, for $p_u, p_v \in H^-$, with $u < v$ and $p_b\in H^+$, the interior points of triangle $(p_u, p_v, p_b)$ are precisely the points of $H^-$ lying between
$p_u$ and $p_v$ in order of their $x$-coordinates and above the line joining $p_u, p_v$. This means that the triangle $(p_b, p_{a_i}, p_{a_{i+1}})$ is empty in $H$, for all $1 \leq i \leq r - 2$ (see Figure~\ref{fig:horton-visible-fan}) and, hence, the $r$-gon $(p_b, p_{a_1}, p_{a_2},\ldots, p_{a_{r-1}})$ is empty in $H$. 
\end{proof}

\begin{figure}[h!]
\centering
\begin{tikzpicture}[x=.70cm,y=.041cm,font=\small,
  selected/.style={fill=hred,draw=white,line width=.3pt},
  other/.style={fill=hblue,draw=white,line width=.3pt}]  
\foreach \i/\y in {
  1/0,2/91,3/14,4/105,5/4,6/95,7/18,8/109,
  9/0,10/91,11/14,12/105,13/4,14/95,15/18,16/109}
  \coordinate (h\i) at ({\i-1},\y);
\foreach \j/\i in {1/1,2/3,3/7,4/11,5/15}
  \coordinate (a\j) at (h\i);
\coordinate (b) at (h8);

\fill[hred!8] (b)--(a1)--(a2)--(a3)--(a4)--(a5)--cycle;
\fill[hred!15] (b)--(a2)--(a3)--cycle;
\fill[hred!15] (b)--(a4)--(a5)--cycle;
\foreach \j in {2,3,4}
  \draw[hred!65,dashed,line width=.65pt] (b)--(a\j);
\draw[hred!80!black,line width=.9pt] (a1)--(b)--(a5);
\draw[hred,line width=1.4pt] (a1)--(a2)--(a3)--(a4)--(a5);

\foreach \i in {2,4,5,6,9,10,12,13,14,16}
  \path[other] (h\i) circle[radius=2.3pt];
\foreach \i in {1,3,7,8,11,15}
  \path[selected] (h\i) circle[radius=2.6pt];
\foreach \j in {1,...,5}
  \node[below=4pt,text=hred!85!black] at (a\j) {$p_{a_{\j}}$};
\node[above=4pt,text=hred!85!black] at (b) {$p_b$};
\node[anchor=west] at (15.6,100) {$H^+$};
\node[anchor=west] at (15.6,9) {$H^-$};
%
%
\end{tikzpicture}
\caption{\small{ A 2-colored Horton set of size 16. The upper-visible
chain $p_{a_1},\ldots, p_{a_5}$ and the apex vertex $p_b$ span an empty, monochromatic hexagon. } }
\label{fig:horton-visible-fan}
\end{figure}
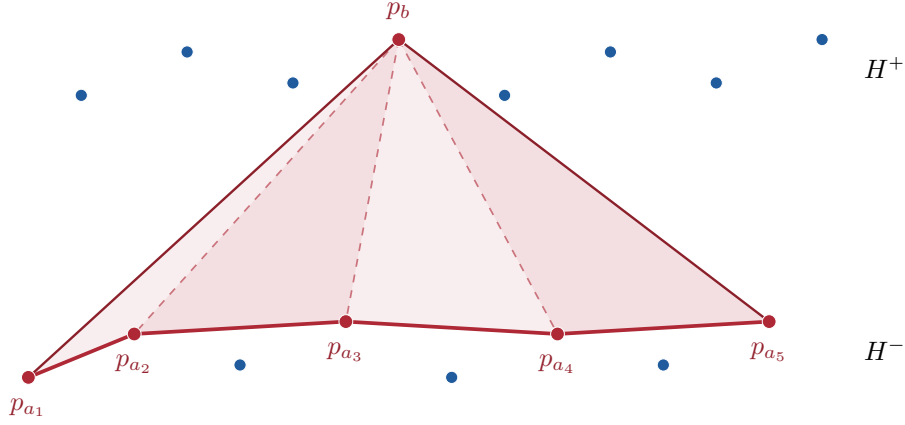

If $|H|\geq4^{r-1}=2^{2r-2}$, each child of $H$ has at least $2^{2r-3}$ points. By Lemma~\ref{lem:chain}, $H^{-}$ contains a monochromatic increasing upper-visible chain $p_{a_1},\ldots, p_{a_{r-1}}$. Without loss of generality, suppose the points of this chain have color 1. If $H^+$ has a point $p_b$ of color 1, by Observation~\ref{obs:rtriangle},  $(p_b, p_{a_1} ,\ldots, p_{a_{r-1}})$ is an empty $r$-gon. If not, all points of $H^+$ are colored 2. Since $2^{2r-3}\geq r$ for $r\geq3$, it contains $r$ consecutive points. Their convex hull contains
no other points of $H$, hence a simple polygonization of these points is a monochromatic empty $r$-gon.  \hfill  $\Box$

\end{document}